\documentclass[10pt]{amsart}
\usepackage[latin1]{inputenc}
\usepackage{amsmath}
\usepackage{amssymb}
\usepackage{latexsym}
\usepackage{verbatim}
\usepackage{graphics}
\usepackage{multirow}
\usepackage{amsfonts}
\usepackage{graphicx}
\usepackage{epstopdf}
\usepackage[margin=0.5in]{geometry}

\usepackage{etoolbox} 
\let\bbordermatrix\bordermatrix 
\patchcmd{\bbordermatrix}{8.75}{4.75}{}{}
\patchcmd{\bbordermatrix}{\left(}{\left[}{}{}
\patchcmd{\bbordermatrix}{\right)}{\right]}{}{}

\usepackage{blkarray}
\usepackage{multirow}

\theoremstyle{plain}
\newtheorem{thm}{Theorem}

\newtheorem{prop}{Proposition}
\newtheorem{lem}{Lemma}
\newtheorem{cor}{Corollary}

\newtheorem{conj}{Conjecture}

\newcommand{\abs}[1]{\left\vert #1 \right\vert}

\def\e{\varepsilon}
\def\d{\delta}

\theoremstyle{definition}
\newtheorem{example}{Example}

\theoremstyle{remark}

\newtheorem*{rem}{Remark}

\begin{document}

\title{The Alexander polynomial for Virtual Twist Knots}

\author{Isaac Benioff}
\address{Williams College}
\email{isaac.benioff@gmail.com}

\author{Blake Mellor}
\address{Department of Mathematics, Loyola Marymount University}
\email{blake.mellor@lmu.edu}

\begin{abstract}
We define a family of virtual knots generalizing the classical twist knots.  We develop a recursive formula for the Alexander polynomial $\Delta_0$ (as defined by Silver and William \cite{sw}) of these virtual twist knots.  These results are applied to provide evidence for a conjecture that the odd writhe of a virtual knot can be obtained from $\Delta_0$.
\end{abstract}

\date{\today}

\maketitle

\section{Introduction}

Since the introduction of virtual knots by Kauffman \cite{ka} as a generalization of classical knot theory, there has been considerable work in the area (see, for example, the book-length survey by Manturov and Ilyutko \cite{mi}). As with classical knots, one of the primary areas of research on virtual knots is defining invariants that will distinguish virtual knots. Often, these invariants are generalizations of classical knot invariants; however, the generalizations are often more complex than the original classical invariants.

In this paper, we focus on the virtual version of the Alexander polynomial. The classical Alexander polynomial is just the first of a sequence of invariants derived from the Alexander module; Silver and Williams \cite{sw} generalized this construction to define an analogous module for virtual knots, and a series of polynomials $\Delta_i(K)(u,v)$ for $i\geq 0$.  For a classical knot, $\Delta_0$ is always trivial, and the polynomial $\Delta_1$ is equal to the classical Alexander polynomial evaluated at $uv$. For virtual knots, however, $\Delta_0$ is \emph{not} trivial, and so provides a tool for distinguishing among virtual knots (and distinguishing virtual knots from classical knots).

While the invariant $\Delta_0$ defined by Silver and Williams (and, in somewhat different form, by Sawollek \cite{sa}) is straightforward to compute for specific examples, it has not been computed for many infinite families of virtual knots. In part, this is because it does not satisfy the same nice skein relation as the classical Alexander polynomial, so the computations become much more complex. In this paper we will provide a recursive formula for computing $\Delta_0$ for an infinite family of \emph{virtual twist knots} which generalize the classical twist knots. 

This formula allows us to use the virtual twist knots as test cases for conjectures about the behavior of $\Delta_0$ in general.  In particular, we use it to test a conjectured relation between $\Delta_0$ and the \emph{odd writhe} of Kauffman \cite{ka2}. Namely, $2\abs{\overline{\Delta}_0(K)(-1,-1)} = \abs{OW(K)}$, where $\overline{\Delta}_0$ is a particular factor of $\Delta_0$, and $OW(K)$ is the odd writhe of $K$.  Using our recursive formula, we are able to prove the relationship for the virtual twist knots, and we conjecture that it holds for all virtual knots.

\section{Virtual knots and the Alexander polynomial}

\subsection{Virtual knots.} Our approach to virtual knots will be combinatorial. Kauffman \cite{ka} showed that virtual knots can be defined as equivalence classes of diagrams modulo certain moves, generalizing the Reidemeister moves of classical knot theory. Diagrams for virtual knots contain both classical crossings (positive and/or negative crossings, if the knot is oriented) and \emph{virtual} crossings, as shown in Figure \ref{F:crossings}.  Two diagrams are equivalent if they are related by a sequence of the Reidemeister moves shown in Figure \ref{F:reidemeister}. Note that moves (I)--(III) are the classical Reidemeister moves. Kauffman \cite{ka} showed that classical knots are equivalent by this expanded set of Reidemeister moves if and only if they are equivalent by the classical Reidemeister moves, so classical knot theory embeds inside virtual knot theory.

\begin{figure}[htbp]
\begin{center}
\scalebox{.6}{\includegraphics{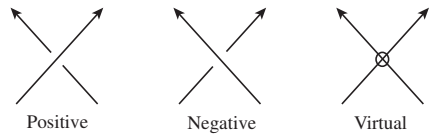}}
\end{center}
\caption{Classical and virtual crossings}
\label{F:crossings}
\end{figure}

\begin{figure}[htbp]
\begin{center}
\scalebox{.8}{\includegraphics{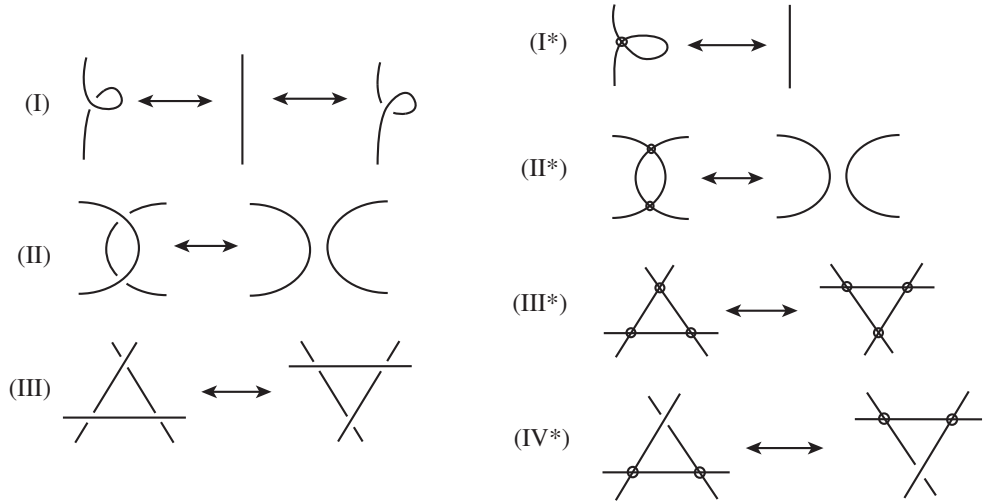}}
\end{center}
\caption{Reidemeister moves for virtual knots}
\label{F:reidemeister}
\end{figure}

\subsection{Alexander polynomial}

Our definition of the Alexander polynomial $\Delta_0(u,v)$ follows Silver and Williams \cite{sw}. Given a virtual knot diagram $D$ with $n$ classical crossings, labeled from $c_1$ to $c_n$, an \emph{arc} of the diagram extends from one classical crossing to the next classical crossing (ignoring any virtual crossings).  Note that these go from crossing to crossing, \emph{not} undercrossing to undercrossing (which is the usual notion of an arc in a classical knot diagram).  So $D$ has $2n$ arcs, which we label from $a_1$ to $a_{2n}$. At each classical crossing, we define two relations among the four arcs incident to that crossing.  The relations depend on whether the crossing is positive or negative, as shown in Figure \ref{F:alexander}.  The result is a system of $2n$ linear equations in $2n$ variables (the arcs).  We call the coefficient matrix for this system the {\it Alexander matrix} for the knot diagram.  We define $\Delta_0(D)(u,v)$ as the determinant of the Alexander matrix. Note that changing the labeling of the arcs will permute the columns of the matrix, and can change the sign of $\Delta_0(D)$ (on the other hand, reordering the crossings permutes \emph{pairs} of rows, so does not change the determinant).

\begin{figure}[htbp]
\begin{center}
\scalebox{.8}{\includegraphics{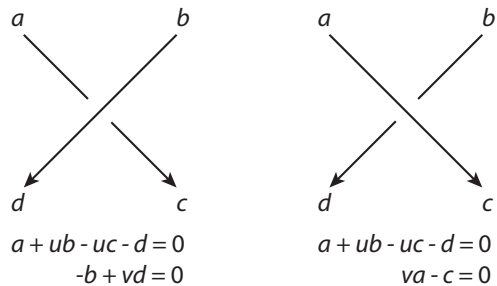}}
\end{center}
\caption{Relations at a positive (left) and negative (right) crossing}
\label{F:alexander}
\end{figure}

To extend $\Delta_0$ from an invariant of diagrams to an invariant of virtual knots, we need to see how it is affected by the Reidemeister moves.  Since we only get relations at the classical crossings, moves (I*)--(IV*) have no effect (these moves do not change the arrangement of classical crossings). Silver and Williams \cite{sw} analyzed the effect of Reidemeister moves (I)--(III), but only modulo the effect of permuting the labels on the arcs.  In our calculations, we will need to keep track of these effects, so we are going to look more carefully at Reidemeister moves (I) and (II) (we do not use move (III) in our calculations). In fact, we discovered that Silver and Williams mistakenly claim that a Reidemeister (II) move does not affect $\Delta_0$; this error is corrected in Lemma \ref{L:R2}.  First, however, we analyze the effect of a Reidemeister (I) move. As in \cite{sw}, we distinguish four types of Reidemeister (I) move, shown in Figure \ref{F:r1alexander}.  Notice that we are fixing the labeling of the arcs; in our proofs, we will need to account for differences in the labelings, but this is not difficult.

\begin{figure}[htbp]
\begin{center}
\scalebox{.6}{\includegraphics{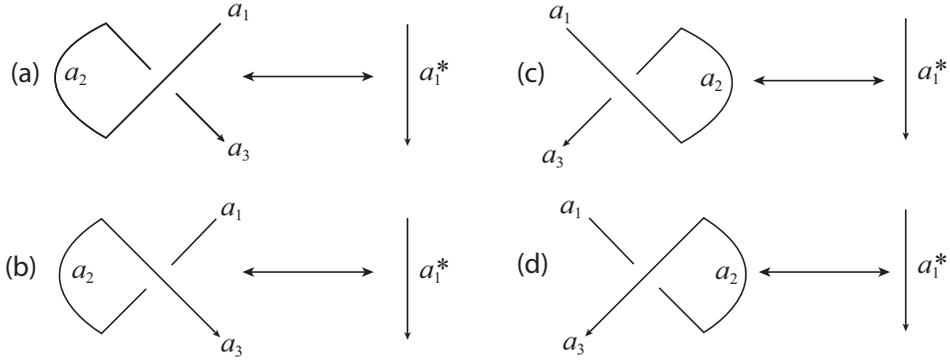}}
\end{center}
\caption{Four types of Reidemeister (I).  In all these diagrams, $a_1, a_2, a_3$ are the first three arcs.}
\label{F:r1alexander}
\end{figure}

\begin{lem}\label{L:R1}
If $D$ is a virtual knot diagram, and $D'$ is the result of applying a Reidemeister (I) move to remove a crossing as in Figure \ref{F:r1alexander}, then: \begin{itemize}
\item $\Delta_0(D) = (uv)\Delta_0(D')$ for a (Ia) or (Ib) move, and 
\item $\Delta_0(D) = (-1)\Delta_0(D')$ for a (Ic) or (Id) move.
\end{itemize}
\end{lem}

\begin{proof}
We will prove the result for a move of type (Ia); the proofs for the other cases are similar.  The crossing shown in Figure \ref{F:r1alexander}(a) gives two Alexander relations:
$$a_2 + ua_1 - ua_3 - a_2 = 0 \implies ua_1 - ua_3 = 0$$
$$-a_1 + va_2 = 0$$
Since we are assuming $a_1, a_2, a_3$ are the first three arcs, the Alexander matrix $M$ for $D$ is:
$$M = \begin{blockarray}{cccc}
    a_1 & a_2 & a_3 & \\
    \begin{block}{(ccc|c)}
    	u & 0 & -u & \bf{0} \\
    	-1 & v & 0 & \bf{0} \\ \cline{1-4}
    	 \beta & \bf{0} & \gamma & \bf{A} \\
    \end{block}
  \end{blockarray}$$
where $A$ is a submatrix and $\beta$ and $\gamma$ are column vectors. Expanding along the first row, we can calculate the determinant:

$$
\begin{aligned}
\det M &= u \det \left ( \begin{array}{cc|c}
v & 0 & \bf{0} \\ \hline
\bf{0} & \gamma & \bf{A}
\end{array}\right )
-u \det \left ( \begin{array}{cc|c}
-1 & v & \bf{0} \\ \hline
\beta & \bf{0} & \bf{A}
\end{array}\right ) \\
& = uv \det\left ( \begin{array}{c|c}
\gamma & \bf{A}
\end{array}\right ) 
+ u \det \left ( \begin{array}{c|c}
\bf{0} & \bf{A}
\end{array}\right )
+ uv \det \left ( \begin{array}{c|c}
\beta & \bf{A}
\end{array}\right ) \\
& = uv \det \left ( \begin{array}{c|c}
\beta+\gamma & \bf{A}
\end{array}\right )
\end{aligned}
$$ \medskip

\noindent Now, notice that in the new diagram, the new strand, $a_1^*$, is involved in each crossing that the strands $a_1$ and $a_3$ were involved in. Therefore, the Alexander matrix $M'$ for $D'$ is:
$$M' = \begin{blockarray}{cc}
    a_1^* & \\
    \begin{block}{(c|c)}
    	\beta+\gamma & \bf{A} \\
    \end{block}
  \end{blockarray}$$
So as desired, $\det{M} = uv\det{M'}$. A similar argument holds for type (Ib), (Ic) and (Id) moves.
\end{proof}

Now we turn to the Reidemeister (II) move.  Again, we distinguish four types of Reidemeister (II) move, as shown in Figure \ref{F:r2alexander}. Again, we give the arcs involved in the move the first 6 labels in the diagram, with odd labels along one strand and even labels along the other (the choice of labels is made to simplify some of our later computations).

\begin{figure}[htbp]
\begin{center}
\scalebox{.8}{\includegraphics{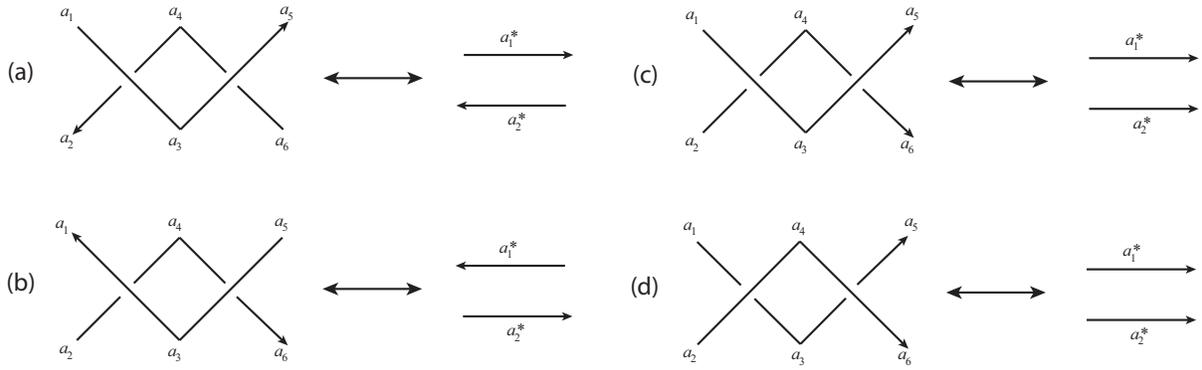}}
\end{center}
\caption{Four types of Reidemeister (II) moves. We assume $a_1, \dots, a_6$ are the first 6 arcs.}
\label{F:r2alexander}
\end{figure}

\begin{lem}\label{L:R2}
If $D$ is a virtual knot diagram, and $D'$ is the result of applying a Reidemeister (II) move to remove two crossings as in Figure \ref{F:r2alexander}, then $\Delta_0(D) = (-uv)\Delta_0(D')$ (for all four types of Reidemeister (II) move).
\end{lem}

\begin{proof}
We will prove the result for moves of type (IIa); the proofs for the others are similar.  The Alexander matrix $M$ for $D$ is:
$$M = \begin{blockarray}{ccccccc}
    a_1 & a_2 & a_3 & a_4 & a_5 & a_6 & \\
    \begin{block}{(cccccc|c@{\hspace*{5pt}})}
     1 & -1 & -u & u & 0 & 0 & {\bf 0} \\
     v & 0 & -1 & 0 & 0 & 0 & {\bf 0} \\
     0 & 0 & u & -u & -1 & 1 & {\bf 0} \\
     0 & 0 & -1 & 0 & v & 0 & {\bf 0} \\
  	\cline{1-7}
   {\bf \alpha} &  {\bf \beta} &  {\bf 0} &  {\bf 0} &  {\bf \gamma} &  {\bf \delta} & {\bf A} \\
    \end{block}
  \end{blockarray}$$
If we add the third row to the first and subtract the fourth row from the second, we get:

$$\det M = \det \begin{blockarray}{ccccccc}
    a_1 & a_2 & a_3 & a_4 & a_5 & a_6 & \\
    \begin{block}{(cccccc|c@{\hspace*{5pt}})}
     1 & -1 & 0& 0 & -1 & 1 & {\bf 0} \\
     v & 0 & 0 & 0 & -v & 0 & {\bf 0} \\
     0 & 0 & u & -u & -1 & 1 & {\bf 0} \\  
     0 & 0 & -1 & 0 & v & 0 & {\bf 0} \\
  	\cline{1-7}
   {\bf \alpha} &  {\bf \beta} &  {\bf 0} &  {\bf 0} &  {\bf \gamma} &  {\bf \delta} & {\bf A} \\
    \end{block}
  \end{blockarray}$$
Taking the cofactor expansion first along column $a_4$, and then along column $a_3$ gives:

$$\det M = (u)(-1) \det 
    \left( \begin{array}{cccc|c@{\hspace*{5pt}}}
     1 & -1 &  -1 & 1 & {\bf 0} \\
     v & 0 &  -v & 0 & {\bf 0} \\
  	\hline
   {\bf \alpha} &  {\bf \beta} &  {\bf \gamma} &  {\bf \delta} & {\bf A} \\
    \end{array} \right)
  $$
Now, expanding along the second row gives:
  
  $$\begin{aligned}
  \det M &= (uv) \det \left( 
      \begin{array}{ccc|c@{\hspace*{5pt}}}
     -1 &  -1 & 1 & {\bf 0} \\
     \hline
   {\bf \beta} &  {\bf \gamma} &  {\bf \delta} & {\bf A} \\
    \end{array} \right)
  -(uv)\det \left(
    \begin{array}{ccc|c@{\hspace*{5pt}}}
     1 &  -1 & 1 & {\bf 0} \\
    \hline
   {\bf \alpha} &  {\bf \beta} &  {\bf \delta} & {\bf A} \\
    \end{array} \right)\\
  &= (uv) \left((-1)\det \left( 
      \begin{array}{cc|c@{\hspace*{5pt}}}
    {\bf \gamma} &  {\bf \delta} & {\bf A} \\
    \end{array}\right)
  + \det \left( 
      \begin{array}{cc|c@{\hspace*{5pt}}}
    {\bf \beta} &  {\bf \delta} & {\bf A} \\
    \end{array} \right)
  + \det \left( 
      \begin{array}{cc|c@{\hspace*{5pt}}}
    {\bf \beta} &  {\bf \gamma} & {\bf A} \\
    \end{array} \right)\right)\\
  &\quad -(uv)\left(\det \left( 
      \begin{array}{cc|c@{\hspace*{5pt}}}
    {\bf \beta} &  {\bf \delta} & {\bf A} \\
    \end{array} \right)
  + \det\left( 
      \begin{array}{cc|c@{\hspace*{5pt}}}
    {\bf \alpha} &  {\bf \delta} & {\bf A} \\
    \end{array} \right)
  + \det \left( 
      \begin{array}{cc|c@{\hspace*{5pt}}}
    {\bf \alpha} &  {\bf \beta} & {\bf A} \\
    \end{array} \right)\right)\\
  &= (-uv)\left(\det \left( 
      \begin{array}{cc|c@{\hspace*{5pt}}}
          {\bf \alpha} &  {\bf \beta} & {\bf A} \\
    \end{array} \right)
    + \det \left( 
      \begin{array}{cc|c@{\hspace*{5pt}}}
    {\bf \alpha} &  {\bf \delta}  & {\bf A} \\
    \end{array} \right)
  + \det\left( 
      \begin{array}{cc|c@{\hspace*{5pt}}}
    {\bf \gamma} &  {\bf \beta} & {\bf A} \\
    \end{array} \right)
  + \det\left( 
      \begin{array}{cc|c@{\hspace*{5pt}}}
    {\bf \gamma} &  {\bf \delta} & {\bf A} \\
    \end{array} \right)\right)\\
  &= (-uv) \det\left( 
      \begin{array}{cc|c@{\hspace*{5pt}}}
    {\bf \alpha+\gamma} &  {\bf \beta+\delta} & {\bf A} \\
    \end{array} \right)
  \end{aligned}$$
In $D'$, arc $a_1^*$ is involved in each crossing that the arcs $a_1$ and $a_5$ were involved in in $D$, and similarly arc $a_2^*$ is involved in each crossing that the arcs $a_2$ and $a_6$ were involved in in $D$. Therefore, the Alexander matrix $M'$ for $D'$ is:

$$M' = \begin{blockarray}{ccc}
    a_1^*& a_2^*\\
    \begin{block}{(cc|c@{\hspace*{5pt}})}
    {\bf \alpha+\gamma} &  {\bf \beta+\delta} & {\bf A} \\
    \end{block}
  \end{blockarray}$$
So $\det M = (-uv)\det M'$, as desired. A similar argument shows that the same relation holds for the three other types of Reidemeister (II) moves.
\end{proof} \medskip

From Lemma \ref{L:R1} and Lemma \ref{L:R2}, two diagrams representing the same virtual knot can have Alexander polynomials which differ by a factor of $\pm(uv)^k$. Typically, when we talk about the Alexander polynomial for a \emph{knot}, we normalize it to remove the indeterminacy.  Following Silver and Williams, if $K$ is a virtual knot with diagram $D$, we define:
$$\Delta_0(K)(u,v) = (-1)^r(uv)^{-s}\Delta_0(D)(u,v)$$
where $s$ is the lowest power of $u$ in $\Delta_0(D)(u,v)$, and $(-1)^r$ is the sign of the term in $(uv)^{-s}\Delta_0(D)(u,v)$ with lowest total degree (if there are multiple terms with the same lowest total degree, we choose the one where $u$ has the lowest degree).

One of the most useful properties of the classical Alexander polynomial is that it satisfies a recursive {\it skein relation} that makes it much easier to compute.  Silver and Williams \cite{sw} prove a similar skein relation for $\Delta_0$, but it has an indeterminacy that makes it difficult to use.  In our next lemma, we remove this indeterminacy by considering diagrams rather than knots, and fixing a labeling for the arcs; this will enable us to use the skein relation more easily (at the cost of having to keep careful track of the labels on the arcs). Our proof is essentially the same as in \cite{sw}, but we include it for completeness.

\begin{figure}[htbp]
\begin{center}
\scalebox{.8}{\includegraphics{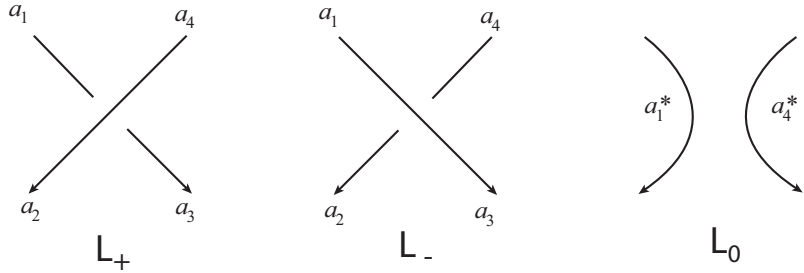}}
\end{center}
\caption{Diagrams in the skein relation of Lemma \ref{L:skein}.}
\label{F:skein}
\end{figure}

\begin{lem}\label{L:skein}
Given labeled diagrams $L_+$, $L_-$ and $L_0$ as shown in Figure \ref{F:skein}, we have
$$\Delta_0(L_+)-\Delta_0(L_-)=(uv-1)\Delta_0(L_0)$$
\end{lem}

\begin{proof}
For $L_+$, the Alexander matrix $M_+$ is:

$$M_+ = \begin{blockarray}{ccccc}
    a_1& a_2 & a_3 & a_4 \\
    \begin{block}{(cccc|c@{\hspace*{5pt}})}
	1 & -1 & -u & u & {\bf 0} \\
	0 & v & 0 & -1  & {\bf 0} \\
	\cline{1-5}
	\alpha & \beta & \gamma & \delta & {\bf A} \\
    \end{block}
  \end{blockarray}$$
Taking the cofactor expansion along the second row, we get:

$$\begin{aligned}
\det M_+ &= v \det \begin{blockarray}{cccc}
    a_1& a_3 & a_4 \\
    \begin{block}{(ccc|c@{\hspace*{5pt}})}
	1 & -u & u & {\bf 0} \\
	\cline{1-4}
	\alpha & \gamma & \delta & {\bf A} \\
    \end{block}
  \end{blockarray}
  - \det\begin{blockarray}{cccc}
    a_1& a_2 & a_4 \\
    \begin{block}{(ccc|c@{\hspace*{5pt}})}
	1 & -1 & -u & {\bf 0} \\
	\cline{1-4}
	\alpha & \beta & \gamma & {\bf A} \\
    \end{block}
  \end{blockarray} \\
  &=v\det \left[\gamma\,\delta\,{\bf A}\right]+uv\det \left[\alpha\,\delta\,{\bf A}\right]+uv\det \left[\alpha\,\gamma\,{\bf A}\right] - \det\left[ \beta\,\gamma\,{\bf A}\right] - \det \left[\alpha\,\gamma\,{\bf A}\right] + u\det \left[\alpha\,\beta\,{\bf A}\right]
  \end{aligned}
$$
For $L_-$, the Alexander matrix $M_-$ is:

$$M_- = \begin{blockarray}{ccccc}
    a_1& a_2 & a_3 & a_4 \\
    \begin{block}{(cccc|c@{\hspace*{5pt}})}
	1 & -1 & -u & u & {\bf 0} \\
	v & 0 & -1 & 0 & {\bf 0} \\
	\cline{1-5}
	\alpha & \beta & \gamma & \delta & {\bf A} \\
    \end{block}
  \end{blockarray}$$
Taking the cofactor expansion along the second row, as with $M_+$, we get:

$$\begin{aligned}
\det M_- &= -v \det \begin{blockarray}{cccc}
    a_2 & a_3 & a_4 \\
    \begin{block}{(ccc|c@{\hspace*{5pt}})}
	-1 & -u & u & {\bf 0} \\
	\cline{1-4}
	\beta & \gamma & \delta & {\bf A} \\
    \end{block}
  \end{blockarray}
  + \det \begin{blockarray}{cccc}
    a_1& a_2 & a_4\\
    \begin{block}{(ccc|c@{\hspace*{5pt}})}
	1 & -1 & u & {\bf 0} \\
	\cline{1-4}
	\alpha & \beta & \delta & {\bf A} \\
    \end{block}
  \end{blockarray} \\
  &= v\det\left[\gamma\,\delta\,{\bf A}\right]-uv\det\left[\beta\,\delta\,{\bf A}\right]-uv\det\left[\beta\,\gamma\,{\bf A}\right]+\det\left[\beta\,\delta\,{\bf A}\right] + \det\left[\alpha\,\delta\,{\bf A}\right] + u\det\left[\alpha\,\beta\,{\bf A}\right]
  \end{aligned}
$$
Then:
\begin{align*}
\Delta_0(L_+)-\Delta_0(L_-)&= \det M_+ - \det M_-\\
&= v\det \left[\gamma\,\delta\,{\bf A}\right]+uv\det \left[\alpha\,\delta\,{\bf A}\right]+uv\det \left[\alpha\,\gamma\,{\bf A}\right] - \det\left[ \beta\,\gamma\,{\bf A}\right] - \det \left[\alpha\,\gamma\,{\bf A}\right] + u\det \left[\alpha\,\beta\,{\bf A}\right] \\
&\quad - v\det\left[\gamma\,\delta\,{\bf A}\right] + uv\det\left[\beta\,\delta\,{\bf A}\right] + uv\det\left[\beta\,\gamma\,{\bf A}\right] - \det\left[\beta\,\delta\,{\bf A}\right] - \det\left[\alpha\,\delta\,{\bf A}\right] - u\det\left[\alpha\,\beta\,{\bf A}\right] \\
&= uv\det\left[\alpha\,\delta\,{\bf A}\right] + uv\det\left[\alpha\,\gamma\,{\bf A}\right] - \det\left[\beta\,\gamma\,{\bf A}\right] - \det\left[\alpha\,\gamma\,{\bf A}\right] \\
&\quad + uv\det\left[\beta\,\delta\,{\bf A}\right] + uv\det\left[\beta\,\gamma\,{\bf A}\right] - \det\left[\beta\,\delta\,{\bf A}\right] - \det\left[\alpha\,\delta\,{\bf A}\right] \\
&= (uv-1)\left(\det\left[\alpha\,\gamma\,{\bf A}\right] + \det\left[\alpha\,\delta\,{\bf A}\right] + \det\left[\beta\,\gamma\,{\bf A}\right] + \det\left[\beta\,\delta\,{\bf A}\right]\right)
\end{align*}
In $L_0$, the arc $a_1^*$ is involved in each crossing that the arcs $a_1$ and $a_2$ were involved in, and similarly the arc $a_4^*$ is involved in each crossing that the arcs $a_3$ and $a_4$ were involved in. Therefore, the Alexander matrix $M_0$ for $L_0$ is:
$$M_0 = \begin{blockarray}{ccc}
    a_1^* & a_4^* \\
    \begin{block}{(ccc@{\hspace*{5pt}})}
    \alpha+\beta & \gamma+\delta & {\bf A} \\
    \end{block}
  \end{blockarray}
$$
So:
\begin{align*}
\Delta_0(L_0) = \det M_0 = \det\left[\alpha+\beta\ \ \gamma+\delta\ \ {\bf A}\right] &= \det\left[\alpha\ \ \gamma+\delta\ \ {\bf A}\right] + \det\left[\beta\ \ \gamma+\delta\ \ {\bf A}\right]\\
&= \det\left[\alpha\ \gamma\ {\bf A}\right] + \det\left[\alpha\ \delta\ {\bf A}\right] + \det\left[\beta\ \gamma\ {\bf A}\right] + \det\left[\beta\ \delta\ {\bf A}\right]
\end{align*}
So as desired, $\Delta_0(L_+)-\Delta_0(L_-)=(uv-1)\Delta_0(L_0)$.
\end{proof}

Finally, we recall two more propositions due to Silver and Williams.

\begin{prop}\label{P:factor}\cite{sw}
Let $L$ be an oriented virtual link. \begin{enumerate}
	\item $(u-1)(v-1)$ divides $\Delta_0(L)(u,v)$.
	\item If $L$ is a knot, then $uv-1$ divides $\Delta_0(L)(u,v)$.
\end{enumerate}
\end{prop} \medskip

\noindent So if $K$ is a knot, $(u-1)(v-1)(uv-1)$ divides $\Delta_0(K)(u,v)$.  We will let $\overline{\Delta}_0(K)(u,v)$ denote the quotient, so
$$\Delta_0(K)(u,v) = (u-1)(v-1)(uv-1)\overline{\Delta}_0(K)(u,v).$$

\begin{prop}\label{P:reverse} \cite{sw}
Given a diagram $D$ of a virtual knot $K$, let $D^\#$ be the result of switching every (classical) crossing of $D$, $D^*$ be the reflection across a vertical line in the plane of the diagram, and $-D$ the result of reversing all orientations.  Let $K^\#$, $K^*$ and $-K$ be the corresponding virtual links.  Then for all $i \geq 0$,
\begin{enumerate}
	\item $\Delta_i(K^\#)(u,v) = -\Delta_i(K)(v,u)$
	\item $\Delta_i(K^*)(u,v) = \Delta_i(K)(u^{-1}, v^{-1})$
	\item $\Delta_i(-K)(u,v) = -\Delta_i(K)(u^{-1}, v^{-1})$
\end{enumerate}
\end{prop}

\section{Virtual Twist Knots}

We define the {\it virtual twist knot} $VT(a_1,\dots,a_n)$ as shown in Figure \ref{F:virtualtwist}.  As compared to a classical twist knot, one of the crossings in the top clasp has been made virtual, and the crossings in the ``twist" have been divided into $n$ blocks of classical half-twists, each separated by a single virtual crossing, where the $i$th block consists of $a_i$ half-twists.  The knot is oriented as shown in Figure \ref{F:virtualtwist}.  If $a_i$ is positive, the crossings in block $i$ have positive sign; if $a_i$ is negative the crossings have negative sign.

\begin{figure}[htbp]
\begin{center}
\scalebox{.8}{\includegraphics{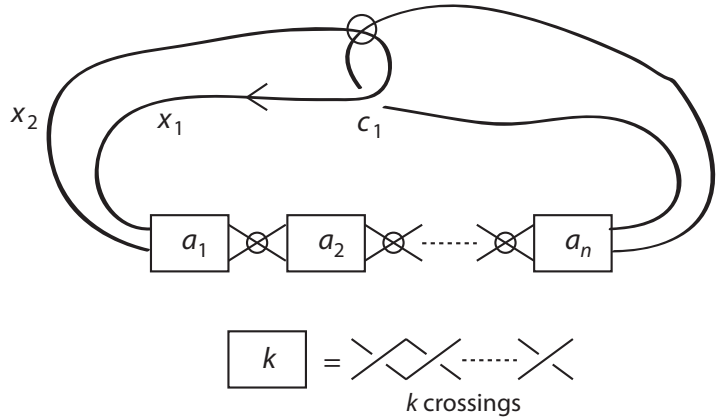}}
\end{center}
\caption{The virtual twist knot $VT(a_1, \dots, a_n)$.}
\label{F:virtualtwist}
\end{figure}

For the purposes of computing the Alexander polynomial, we will fix a labeling of the arcs and crossings of the virtual twist knot.  We label the crossing in the clasp $c_1$ (as in Figure \ref{F:virtualtwist}), and label the other crossings $c_2, c_3, \dots$ from left to right in the twist.  The two arcs to the left of the twist are labeled $x_1$ and $x_2$ as in Figure \ref{F:virtualtwist}.  The arcs along the strand oriented from left to right are given odd labels $x_1, x_3, x_5$, etc., while the arcs along the strand oriented from right to left are given even labels $x_2, x_4, x_6,$ etc. (the subscripts still increase from left to right).  So at crossing $c_{i+1}$, the orientations and labels match one of the two types in Figure \ref{F:crossing} (here they are shown for positive crossings; changing the sign of the crossing does not change the type).  Note that the crossings along the twist alternate between the two types, unless there is a virtual crossing in between.

\begin{figure}[htbp]
\begin{center}
\scalebox{.8}{\includegraphics{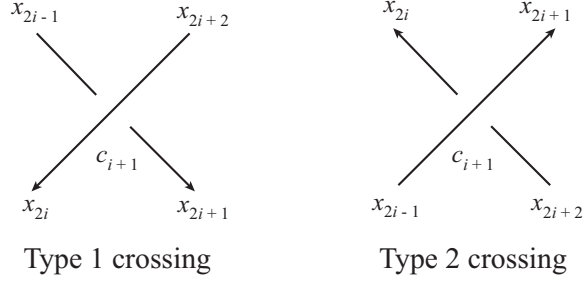}}
\end{center}
\caption{The two types of (positive) crossings in the twist.}
\label{F:crossing}
\end{figure}

\section{Alexander polynomial for Virtual Twist Knots: Special Cases}\label{S:base}

Our goal is to compute the Alexander polynomial $\Delta_0$ for the virtual twist knots. In this section, we will find closed formulas for $\Delta_0$ for four special infinite families of virtual twist knots.  These will be the base cases for the recursive formula we derive in the next section.

\begin{thm}\label{T:base}
Given a sequence of $m$ ones, then:
	\begin{itemize}
		\item[] $\begin{aligned} \Delta_0(VT(1,\dots,1)) &=-1 + u^m + v - u^{m+1}v - u^{m}v^{m+1} + u^{m+1}v^{m+1}\\
		&=(u-1)(v-1)(uv-1)\sum_{i=0}^{m-1}\sum_{j=i}^{m-1}v^{i}u^{j} \end{aligned}$
		\item[] $\begin{aligned} \Delta_0(VT(0,1,\dots,1)) &=(-1)^{m+1}(v - uv - v^{m} + u^{m}v^{m} + uv^{m+1} - u^{m}v^{m+1})\\
		&=(-1)^{m}v(u-1)(v-1)(uv-1)\sum_{i=0}^{m-2}\sum_{j=i}^{m-2}u^{i}v^{j} {\rm \ (or\ 0\ if\ }m\leq1) \end{aligned}$
		\item[] $\begin{aligned} \Delta_0(VT(1,\dots,1,0)) &=-u + u^{m} + uv - u^{m+1}v - u^{m}v^{m} + u^{m+1}v^{m}\\
		&=u(u-1)(v-1)(uv-1)\sum_{i=0}^{m-2}\sum_{j=i}^{m-2}v^{i}u^{j} {\rm \ (or\ 0\ if\ }m\leq1) \end{aligned}$
		\item[] $\begin{aligned} \Delta_0(VT(0,1,\dots,1,0)) &=(-1)^{m+1}(1 - u - v^{m} + u^{m+1}v^{m} + uv^{m+1} - u^{m+1}v^{m+1})\\
		&= (-1)^{m}(u-1)(v-1)(uv-1)\sum_{i=0}^{m-1}\sum_{j=i}^{m-1}u^{i}v^{j} {\rm \ (or\ 0\ if\ }m=0) \end{aligned}$
	\end{itemize}
Here, we are computing the polynomials for the particular \emph{diagram} shown in Figure \ref{F:virtualtwist}; they are not yet normalized to give the polynomial for the knot.
\end{thm}

\begin{proof}
We will prove the theorem for $VT(1,\dots,1)$; the proofs for the other cases are similar.  Define the following matrices:
$$
A=\begin{pmatrix}
1 & -1 \\
0 & v
\end{pmatrix} \quad 
B=\begin{pmatrix}
-u & u \\
0 & -1
\end{pmatrix} \quad
C=\begin{pmatrix}
-1 & u \\
v & -1
\end{pmatrix} \quad
D=\begin{pmatrix}
1 & -u \\
0 & 0
\end{pmatrix}
$$
In $VT(1, \dots, 1)$, the real and virtual crossings alternate along the twist, starting and ending with a real crossing. So at each crossing in the twist, the ``odd" strand is coming down under the ``even" strand.  Hence, excluding the clasp, every crossing will be Type 1 (see Figure \ref{F:crossing}).  So at crossing $c_{i+1}$ we get the submatrix:
$$
\begin{blockarray}{cccc}
    x_{2i-1}& x_{2i} & x_{2i+1} & x_{2i+2} \\
    \begin{block}{(cccc@{\hspace*{5pt}})}
    {1} &  {-1} & {-u} & {u} \\
    {0} &  {v} & {0} & {-1} \\
    \end{block}
  \end{blockarray}=\begin{bmatrix}
\bf{A} & \bf{B} 
\end{bmatrix}
$$
Since the final classical crossing is \emph{not} followed by a virtual crossing, the ``odd" strand leaves the twist as the bottom strand, and the crossing at the clasp will be positive.  So the crossing at the clasp looks like:
$$\scalebox{.8}{\includegraphics{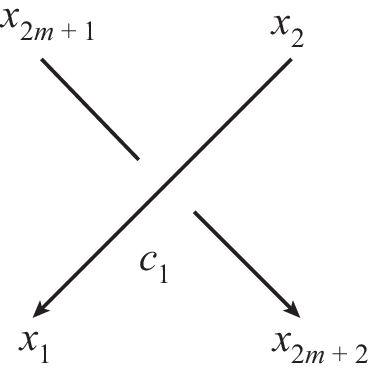}}$$
which gives rise to the submatrix:
$$
\begin{blockarray}{ccccccc}
    x_{1}& x_{2} & & \dots & & x_{2m+1} & x_{2m+2} \\
    \begin{block}{(ccccccc@{\hspace*{5pt}})}
    {-1} &  {u} & 0 & \dots & 0 & {1} & {-u} \\
    {v} &  {-1} & 0 & \dots & 0 & {0} & {0} \\
    \end{block}
  \end{blockarray}=\begin{bmatrix}
\bf{C} & \bf{0} & \cdots & \bf{0} & \bf{D} 
\end{bmatrix}
$$
This gives us the following Alexander matrix for $VT(1,\dots,1)$:
$$
M=\begin{pmatrix}
\mathbf{C} & \mathbf{0} & \mathbf{0} & \mathbf{0} & \mathbf{0} & \cdots & \mathbf{D}\\
\mathbf{A} & \mathbf{B} & \mathbf{0} & \mathbf{0} & \mathbf{0} \\
\mathbf{0} & \mathbf{A} & \mathbf{B} & \mathbf{0} & \mathbf{0}\\
\mathbf{0} & \mathbf{0} & \mathbf{0}  & \ddots & & \\ 
\vdots \\
\mathbf{0} & \mathbf{0} & \mathbf{0} & \mathbf{0} & \cdots & \mathbf{A} & \mathbf{B}
\end{pmatrix}
$$
Right-multiplying $M$ by the elementary block matrix $E(i,N)$ shown below adds the $i$th column of $M$ multiplied by $N$ to the $(i-1)$th column of $M$.
$$
	E(i,N)=\bordermatrix{&c_1&\cdots & c_{i-1} & c_i & \cdots & c_{m+1}\cr
                r_1& \mathbf{I} &  \cdots & \mathbf{0}  & \mathbf{0} & \cdots & \mathbf{0} \cr
                \vdots& \vdots & \ddots & &  && \vdots \cr
                r_{i-1} & \bf{0} &  & \mathbf{I} & \bf{0} &  & \bf{0} \cr
                r_i& \bf{0} &  & \mathbf{N} & \mathbf{I} &  & \bf{0}\cr
                 \vdots & \vdots & & & & \ddots & \vdots \cr
                 r_{m+1} & \mathbf{0}  &  \cdots & \mathbf{0} & \mathbf{0} & \cdots & \mathbf{I}}
$$
In the case of $VT(1,\dots,1)$, this means that right multiplying $M$ by $E(i,-B^{-1}A)$ will cancel the $A$ block in the $i$th row and $(i-1)$th column (assuming $i \geq 2$). Therefore, right-multiplying $M$ by the product:
$$
\prod_{i=0}^{m-1}E(m+1-i,-B^{-1}A) = E(m+1,-B^{-1}A)\cdots E(2,-B^{-1}A) 
$$
will reduce $M$ into upper triangular form while preserving its determinant. The determinant is then the product of the determinants of the block matrices along the diagonal. Except for the entry in the first row and column, every entry along the diagonal is $B$, whose determinant is easily computed to be $u$. Because $B$ shows up $m$ times along the diagonal, we know that $u^{m}$ will be a factor of the determinant. The entry in first row and column, however, will take the form $C+(-1)^{m}D(B^{-1}A)^{m}$. For convenience, define the matrix $X$:
$$
	X=B^{-1}A=\begin{pmatrix}
		-\frac{1}{u}	&	\frac{1}{u}-v	\\
		0	&	-v
	\end{pmatrix}
$$
The entry in the first row and column of the matrix can now be rewritten as $C+(-1)^{m}DX^{m}$, and the problem of computing its determinant is reduced to finding a general form for $X^{m}$. To find such a form, notice that the following similarity relation holds:
$$
	X=PJP^{-1}=\begin{pmatrix}
		1	&	1	\\
		0	&	1
	\end{pmatrix}\begin{pmatrix}
		-\frac{1}{u}	&	0	\\
		0	&	-v
	\end{pmatrix}\begin{pmatrix}
		1	&	-1	\\
		0	&	1
	\end{pmatrix}
$$
Hence,
$$
	X^m=\begin{pmatrix}
		1	&	1	\\
		0	&	1
	\end{pmatrix}\begin{pmatrix}
		\left ( -\frac{1}{u} \right )^m	&	0	\\
		0	&	(-v)^m
	\end{pmatrix}\begin{pmatrix}
		1	&	-1	\\
		0	&	1
	\end{pmatrix}= (-1)^m\begin{pmatrix}
		\frac{1}{u^m} & v^m- \frac{1}{u^m} 	\\
		0	&	v^m
	\end{pmatrix}
$$
It is now straightforward to compute:
$$
\det(C+(-1)^mDX^{m})=1 - \frac{1}{u^m} - u v + \frac{v}{u^m} - v^{m+1} + u v^{m+1}
$$ 
Since the product of the determinants of the other entries on $M$'s diagonal is $u^m$, we see that:
$$
\det{M}= u^m\det(C+(-1)^mDX^{m}) =-1+u^m+v-u^{m+1}v-u^{m}v^{m+1}+u^{m+1}v^{m+1}
$$
Using the identity $x^n - 1 = (x-1)\sum_{i=0}^{n-1}{x^i}$, we can factor $\det{M}$ as:
$$\det{M} = \Delta_0(VT(1,\dots, 1)) =(u-1)(v-1)(uv-1)\sum_{i=0}^{m-1}\sum_{j=i}^{m-1}v^{i}u^{j}$$
A similar argument can be used to compute $\Delta_0$ for $VT(0,1,\dots,1)$, $VT(1,\dots,1,0)$, and $VT(0,1,\dots,1,0)$.
\end{proof}

\section{Alexander polynomial for Virtual Twist Knots: Recursion}\label{S:recursion}

We will use the skein relation of Lemma \ref{L:skein} to find a recursive formula for the Alexander polynomial of the virtual twist knots (more precisely, for the particular diagrams represented in Figure \ref{F:virtualtwist}). Our first step is to find a formula for the Alexander polynomial of the link that results when one of the classical crossings of a virtual twist knot is ``smoothed."  For this section, it is convenient to define the following functions:
\begin{align*}
	 &s(i)=\sum_{j=1}^{i}(a_j+1) \quad {\rm and\ }s(0) = 0\\
	 &p(i)= \left\{ \begin{matrix} 0 & {\rm if\ }i{\rm \ is\ even} \\ 1 & {\rm if\ }i{\rm \ is\ odd} \end{matrix} \right.\\
	 &\d = \sum_{j=1}^n{p(a_j)p(s(j))}\\
	 &\e(i) = (p(s(i-1))-1)+\sum_{j<i}p(a_j)p(s(j))+\sum_{j\geq i}p(a_j)p(1+s(j-1))
\end{align*}

\begin{figure}[htbp]
\begin{center}
\scalebox{.8}{\includegraphics{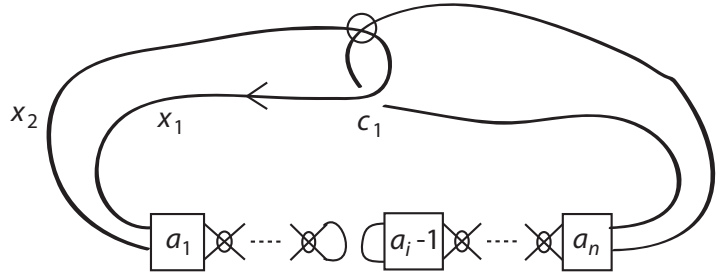}}
\end{center}
\caption{Virtual twist knot with a smoothed crossing.}
\label{F:smoothed}
\end{figure}

\begin{lem}\label{L:smoothed}
Consider the virtual link diagram $L_i$ (as shown in Figure \ref{F:smoothed}) generated by smoothing out the first crossing in the $i$th block in the diagram of the virtual twist knot $VT(a_1,\dots,a_n)$ shown in Figure \ref{F:virtualtwist}. Then 
$$\Delta_0(L_i)=(-uv)^{\sum_{j=1}^n{\left\lfloor \frac{\abs{a_j}}{2} \right\rfloor}}(-1)^{-p(s(i-1))+\d+s(n)}(uv)^{\e(i)}(u-1)(v-1).$$
\end{lem}

\begin{proof}
By using Reidemeister (I) moves to undo half-twists, the diagram for $L_i$ can be reduced to one of the virtual Hopf links, $VHL_+$ or $VHL_-$, shown in Figure \ref{F:hopf}. 

\begin{figure}[htbp]
\begin{center}
\scalebox{.8}{\includegraphics{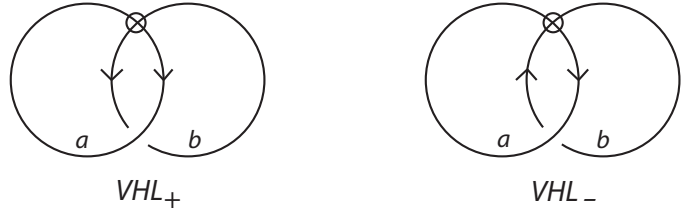}}
\end{center}
\caption{Hopf links $VHL_+$ and $VHL_-$.}
\label{F:hopf}
\end{figure}

Using the labels shown in Figure \ref{F:hopf}, we compute that $\Delta_0(VHL_+)=\left\vert\begin{matrix} u-1& 1-u \\ v-1 & 0 \end{matrix} \right\vert = (u-1)(v-1)$ and $\Delta_0(VHL_-)=\left\vert\begin{matrix} 1-u & u-1 \\ v-1 & 0 \end{matrix} \right\vert = -(u-1)(v-1)$, so the sign of $\Delta_0(L_i)$ (for the given labeling) depends on the sign of the crossing in the clasp of the original twist knot, $VT(a_1,\dots,a_n)$.  This crossing is positive exactly when $\sum{a_i} + (n-1) = s(n)-1$ is odd, or when $s(n)$ is even, so the final Hopf link has $\Delta_0 = (-1)^{s(n)}(u-1)(v-1)$.

\begin{figure}[htbp]
\begin{center}
\scalebox{1}{\includegraphics{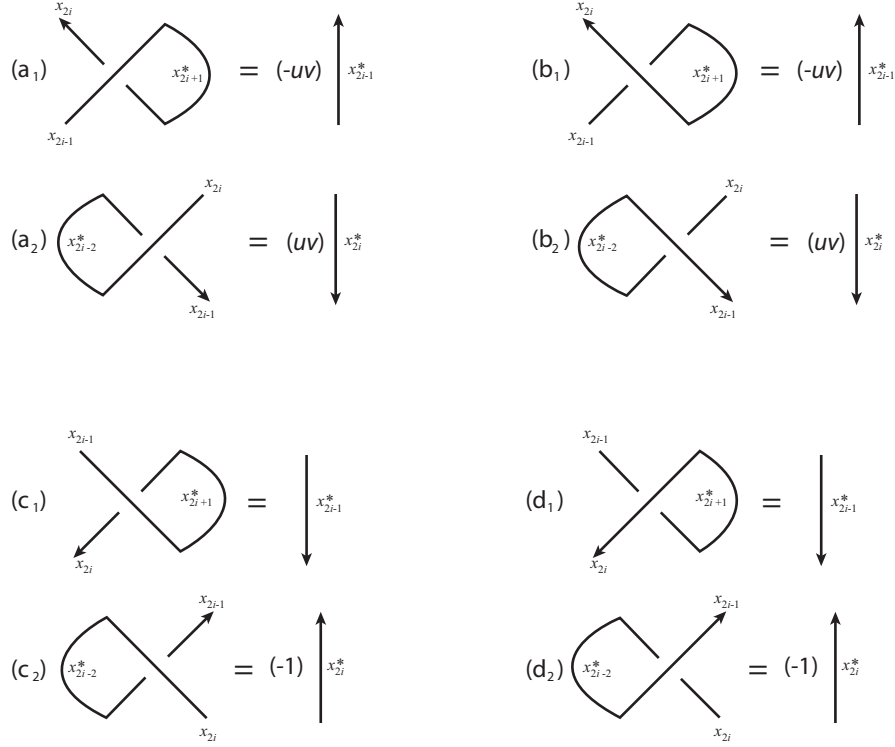}}
\end{center}
\caption{Effect of eight types of Reidemeister 1 moves on $\Delta_0$.}
\label{F:reidemeister1}
\end{figure}

By Lemma \ref{L:R1}, the Reidemeister (I) moves used to undo the half-twists will change $\Delta_0$ as shown in Figure \ref{F:reidemeister1} (the equalities in the figure are between the polynomials of the diagrams shown). To see this, we note that the labelings require us to distinguish two sub-types for each move. If we compare the labels on a move of type (Ia$_1$) to the label on a move of type (Ia) in Figure \ref{F:r1alexander}, there are two differences to account for.  First of all, the labels start at $2i-1$ rather than 1; moving the three columns corresponding to these arcs to the left side of the matrix requires moving each of them past some number $k$ of other columns.  Then the new arc with label $x_{2i-1}^*$ is moved back past the $k$ columns.  Altogether, this changes the determinant by a factor of $(-1)^{4k} = 1$. Secondly, the order of the three columns differs from the order in Figure \ref{F:r1alexander} by a transposition; this changes the determinant by a factor of $-1$.  Hence the effect of a move of type (Ia$_1$) is to multiply the polynomial by $-uv$, rather than by $uv$. Similar reasoning shows that the effects of the other moves are as shown in Figure \ref{F:reidemeister1}.

We first consider removing twists to the right of the smoothed crossing in $L_i$.  Each full twist involves two Reidemeister (I) moves, one of type (Ia$_2$ or Ib$_2$) and one of type (Id$_2$ or Ic$_2$) (depending on whether the crossings in the twist are positive or negative), and hence multiplies $\Delta_0$ by $(uv)(-1) = -uv$.  So in the $j$th block (where $j > i$), if $a_j$ is even then the effect of removing those twists is to multiply $\Delta_0$ by $(-uv)^{\frac{\abs{a_j}}{2}}$.  If $a_j$ is odd then the first half-twist is type (Ia$_2$ or Ib$_2$) if an even number of (real and virtual) crossings preceded the $j$th block, and type (Id$_2$ or Ic$_2$) otherwise.  So the first half-twist contributes $uv$ if $s(j-1)$ is even and $-1$ if $s(j-1)$ is odd. So the final change is to multiply $\Delta_0$ by:
$$
(-uv)^{\left\lfloor \frac{\abs{a_j}}{2} \right\rfloor}(-1)^{p(a_j)p(s(j-1))}(uv)^{p(a_j)p(1+s(j-1))}
$$
Notice that if $a_j$ is even, then $p(a_j)=0$, $\left\lfloor\frac{\abs{a_j}}{2}\right\rfloor=\frac{\abs{a_j}}{2}$, and we get $(-uv)^{\frac{\abs{a_j}}{2}}$. If $a_j$ is odd and $s(j-1)$ even, we get $(-uv)^{\left\lfloor\frac{\abs{a_j}}{2}\right\rfloor}(uv)$, and if $s(j-1)$ is odd, we get $(-uv)^{\left\lfloor\frac{\abs{a_j}}{2}\right\rfloor}(-1)$.

The $i$th block differs slightly in that its first crossing was smoothed out to produce $L$, so we are beginning with its second half-twist. Observe that $\left\lfloor\frac{\abs{a_i}-1}{2}\right\rfloor=\left\lfloor\frac{\abs{a_i}}{2}\right\rfloor-p(\abs{a_i}-1) = \left\lfloor\frac{\abs{a_i}}{2}\right\rfloor-p(a_i-1)$ (since $p(\abs{a_i}-1) = p(a_i-1)$), and that $p(x\pm 1)=1-p(x)$ for any $x$.  Hence, by similar reasoning as before, we see that the result of untwisting the $i$th block is to multiply $\Delta_0$ by:

\begin{align*}
(-uv)^{\left\lfloor \frac{\abs{a_i} -1}{2} \right\rfloor}(-1)^{p(a_i -1)p(s(i-1)+1)}&(uv)^{p(a_i -1)p(s(i-1))} \\
&= (-uv)^{\left\lfloor \frac{\abs{a_i}}{2} \right\rfloor-p(a_i-1)}(-1)^{p(a_i -1)p(s(i-1)+1)}(uv)^{p(a_i -1)p(s(i-1))}\\
&= (-uv)^{\left\lfloor \frac{\abs{a_i}}{2} \right\rfloor}(-1)^{p(a_i -1)p(s(i-1)+1)-p(a_i-1)}(uv)^{p(a_i -1)p(s(i-1)) - p(a_i-1)}\\
&=(-uv)^{\left\lfloor\frac{\abs{a_i}}{2}\right\rfloor}(-1)^{p(a_i-1)(p(s(i-1)+1)-1)}(uv)^{p(a_i-1)(p(s(i-1))-1)}\\
&=(-uv)^{\left\lfloor\frac{\abs{a_i}}{2}\right\rfloor}(-1)^{(1-p(a_i))(-p(s(i-1)))}(uv)^{(1-p(a_i))(-p(s(i-1)+1))}\\
&=(-uv)^{\left\lfloor \frac{\abs{a_i}}{2} \right\rfloor}(-1)^{(p(a_i)-1)p(s(i-1))}(uv)^{(p(a_i) -1)p(s(i-1)+1)}
\end{align*}

We must also remove any half-twists that remain to the left of the $i$th block, using moves of type (Ia$_1$ or Ib$_1$) and (Id$_1$ or ic$_1$). Once again, removing a full twist changes $\Delta_0$ by a factor of $-uv$.  In this case, only the type (Ia$_1$ or Ib$_1$) moves change $\Delta_0$, so when a block has an odd number of crossings we need to determine which type of move is needed to remove the extra crossing. We look specifically at the first crossing in the block.  A type (Ia$_1$ or Ib$_1$) move is used to remove a type 2 crossing, which occurs as the first crossing in the block when $s(j-1)$ is odd. Notice that if $s(j-1)$ is odd and $a_j$ is odd, then $p(s(j-1))=p(s(j)) = 1$. So the change to $\Delta_0$ from removing the twists in block $j$, where $j < i$, is to multiply by:
$$
(-uv)^{\left\lfloor \frac{\abs{a_j}}{2} \right\rfloor}(-uv)^{p(a_j)p(s(j-1))}=(-uv)^{\left\lfloor \frac{\abs{a_j}}{2} \right\rfloor}(-uv)^{p(a_j)p(s(j))}=(-uv)^{\left\lfloor \frac{\abs{a_j}}{2} \right\rfloor}(-1)^{p(a_j)s(j)}(uv)^{p(a_j)s(j)}
$$
Combining all of these results with the value of $\Delta_0$ for the virtual Hopf link gives:

$$
\begin{aligned}
\Delta_0(L)=&(-uv)^{\sum_{j = 1}^n{\left\lfloor\frac{\abs{a_j}}{2}\right\rfloor}}(-1)^{(p(a_i)-1)p(s(i-1))+\sum_{j<i}p(a_j)p(s(j))+\sum_{j>i}p(a_j)p(s(j-1))}\\
&\cdot(uv)^{(p(a_i)-1)p(s(i-1)+1)+\sum_{j<i}p(a_j)p(s(j))+\sum_{j>i}p(a_j)p(1+s(j-1))}\\
&\cdot(-1)^{s(n)}(u-1)(v-1)
\end{aligned}
$$
Observe that if $a_j$ is even, then $p(a_j)p(s(j-1))=p(a_j)p(s(j))=0$, and if $a_j$ is odd, then $p(s(j-1))=p(s(j))$, so $p(a_j)p(s(j-1))=p(a_j)p(s(j))$ for any $a_j$. Also, $p(s(j)+1) = 1-p(s(j))$. We can then reduce our expression for $\Delta_0(L)$ to:

$$
\begin{aligned}
\Delta_0(L)=&(-uv)^{\sum_{j = 1}^n{\left\lfloor\frac{\abs{a_j}}{2}\right\rfloor}}(-1)^{-p(s(i-1))+\sum_{j=1}^np(a_j)p(s(j))}\\
&\cdot(uv)^{(p(s(i-1))-1)+\sum_{j<i}p(a_j)p(s(j))+\sum_{j\geq i}p(a_j)p(1+s(j-1))}\\
&\cdot(-1)^{s(n)}(u-1)(v-1)\\
=&(-uv)^{\sum_{j = 1}^n{\left\lfloor\frac{\abs{a_j}}{2}\right\rfloor}}(-1)^{-p(s(i-1))+\d}(uv)^{\e(i)}(-1)^{s(n)}(u-1)(v-1)\\
=&(-uv)^{\sum_{j = 1}^n{\left\lfloor\frac{\abs{a_j}}{2}\right\rfloor}}(-1)^{-p(s(i-1))+s(n)+\d}(uv)^{\e(i)}(u-1)(v-1)
\end{aligned}
$$
\end{proof}

We use Lemma \ref{L:smoothed} to derive a recursive formula for the Alexander polynomial of a twist knot $VT(a_1, \dots, a_n)$.  We will derive the formula for the case when all $a_i$'s are nonnegative; we will discuss later how to modify the formula when some of the blocks of crossings are negative.

\begin{thm}\label{T:recursion}
Given the diagram of the virtual twist knot $VT(a_1,\dots,a_n)$ shown in Figure \ref{F:virtualtwist}, then:
$$ 
\overline{\Delta}_0(VT(a_1,\dots,a_n))(u,v)=(-uv)^{\sum_{i=1}^{n}\left\lfloor\frac{a_i}{2}\right\rfloor}\left(\overline{\Delta}_0(VT(p(a_1),\dots,p(a_n)))(u,v)+\sum_{i=1}^{n}\left\lfloor\frac{a_i}{2}\right\rfloor(-1)^{\d+s(n)}(uv)^{\e(i)}\right)
$$
where for a given knot $K$, $\Delta_0(K)(u,v)=(u-1)(v-1)(uv-1)\overline{\Delta}_0(K)(u,v)$.
\end{thm}
\begin{proof}
We are going to use the skein relation of Lemma \ref{L:skein}.  Let $L_+=VT(p(a_1),\dots,p(a_{i-1}),a_i,\dots,a_n)$, and suppose $a_i \geq 2$. All the crossings in the twists are positive, so choose the first crossing in the $i$th block to be the crossing we change to construct $L_-$ and $L_0$. In $L_-$, changing the sign of the crossing gives us two crossings that we can remove by a Reidemeister (II) move, leaving us with a diagram for $VT(p(a_1),\dots,p(a_{i-1}),a_i-2,a_{i+1},\dots,a_n)$. Figure \ref{F:reidemeister2} shows the two possible diagrams for $L_-$, depending on whether the first crossing in the $i$th block was type 1 or type 2.  In either case, using Lemma \ref{L:R2}, removing the two crossings changes $\Delta_0$ by a factor of $-uv$.  To see this, we first move the columns corresponding to arcs $x_{2r-1}$ through $x_{2r+4}$ to the left of the Alexander matrix; since there are 6 columns, this requires an even number of transpositions, so the determinant remains unchanged.  If the crossing changed was type 1 (as on the left of Figure \ref{F:reidemeister2}), then we apply a move of type (IIa) (see Figure \ref{F:r2alexander}), and $\Delta_0$ changes by a factor of $-uv$.  If the crossing was type 2, then we apply a move of type (IIb), only with the odd and even labels reversed.  Since there are three such pairs of labels before the move, and one pair after the move, reversing these labels does not change the sign of the determinant of the Alexander matrix either before or after the move, so again the effect is to multiply $\Delta_0$ by $-uv$.  So $\Delta_0(L_-) = (-uv)\Delta_0(VT(p(a_1),\dots,p(a_{i-1}),a_i-2,a_{i+1},\dots,a_n))$.

\begin{figure}[htbp]
\begin{center}
\scalebox{.6}{\includegraphics{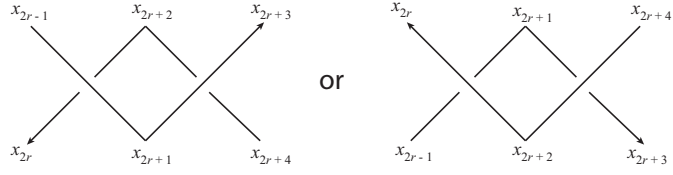}}
\end{center}
\caption{Possible diagrams for $L_-$.}
\label{F:reidemeister2}
\end{figure}

Now we turn to $L_0$.  If the crossing we smooth is type 1, then we are in the situation of Lemma \ref{L:skein}, and $\Delta_0(L_+) = \Delta_0(L_-) + (uv-1)\Delta_0(L_0)$.  However, if the crossing is type 2, we need to permute the labels to apply Lemma \ref{L:skein}.  In particular, in Figure \ref{F:skein}, we would need to interchange $a_1$ and $a_4$, $a_2$ and $a_3$, and $a_1^*$ and $a_4^*$.  These effect of these three transpositions is to change the sign of $\Delta_0(L_0)$, so we have $\Delta_0(L_+) = \Delta_0(L_-) - (uv-1)\Delta_0(L_0)$.  Notice that the first crossing of the $i$th block is type 1 exactly when $p(s(i-1)) = 0$.

So:
\begin{align*}
\Delta_0(VT(p(a_1),\dots,p(a_{i-1}),a_i,\dots,a_n)) = &(-uv)\Delta_0(VT(p(a_1),\dots,p(a_{i-1}),a_i-2,a_{i+1},\dots,a_n))\\
	&+ (-1)^{p(s(i-1))}(uv-1)\Delta_0(L_0)
\end{align*}

Now we can substitute the formula we found for $\Delta_0(L_0)$ in Lemma \ref{L:smoothed}.  Notice that the factor $(-1)^{-p(s(i-1))}$ in the expression for $\Delta_0(L_0)$ exactly cancels the factor $(-1)^{p(s(i-1))}$ above.  Also, note that $\left\lfloor \frac{p(a_j)}{2} \right\rfloor = 0$ for $j \leq i-1$. Hence:
\begin{align*}
\begin{split}
\Delta_0(VT(p(a_1),\dots,p(a_{i-1}),a_i,\dots,a_n)) = &(-uv)\Delta_0(VT(p(a_1),\dots,p(a_{i-1}),a_i-2,a_{i+1},\dots,a_n))\\&+(-uv)^{\sum_{j\geq i}{\left\lfloor \frac{a_j}{2} \right\rfloor}}(-1)^{\d+s(n)}(uv)^{\e(i)}(uv-1)(u-1)(v-1)
\end{split}
\end{align*}
We can then apply the skein relation to $VT(p(a_1),\dots,p(a_{i-1}),a_i-2,a_{i+1},\dots,a_n)$ in the same way. Again, changing the first crossing in the $i$th block to a negative crossing allows us to reduce $L_-$ to $VT(p(a_1),\dots,p(a_{i-1}),a_i-4,a_{i+1},\dots,a_n)$ by a Reidemeister (II) move which gives us the relation:
$$
	\Delta_0(L_-)=(-uv)\Delta_0(VT(p(a_1),\dots,p(a_{i-1}),a_i-4,a_{i+1},\dots,a_n))
$$
Because the parity of each block is the same as it was in $VT(p(a_1),\dots,p(a_{i-1}),a_i,\dots,a_n)$, we know that the number of Reidemeister (I) moves of each type used to reduce $L_0$ to a virtual Hopf link will remain unchanged except for the fact that the $i$th block now has size $a_i-2$, so you have one fewer factor of $-uv$. Thus, we see that:
\begin{align*}
		\Delta_0(VT(p(a_1),\dots,p(a_{i-1}),a_i-2,a_{i+1},\dots,a_n))=&(-uv)\Delta_0(VT(p(a_1),\dots,p(a_{i-1}),a_i-4,a_{i+1},\dots,a_n))\\&+(-uv)^{-1+\sum_{j\geq i}{\left\lfloor \frac{a_j}{2} \right\rfloor}}(-1)^{\d+s(n)}(uv)^{\e(i)}(uv-1)(u-1)(v-1)
\end{align*}
and so
\begin{align*}
	\Delta_0(VT(p(a_1),\dots,p(a_{i-1}),a_i,\dots,a_n)) =&(-uv)\Delta_0(VT(p(a_1),\dots,p(a_{i-1}),a_i-2,a_{i+1},\dots,a_n))\\
	&+(-uv)^{\sum_{j\geq i}{\left\lfloor \frac{a_j}{2} \right\rfloor}}(-1)^{\d+s(n)}(uv)^{\e(i)}(uv-1)(u-1)(v-1)	\\
	 =&(-uv)^{2}\Delta_0(VT(p(a_1),\dots,p(a_{i-1}),a_i-4,a_{i+1},\dots,a_n))\\
	 &+2(-uv)^{\sum_{j\geq i}{\left\lfloor \frac{a_j}{2} \right\rfloor}}(-1)^{\d+s(n)}(uv)^{\e(i)}(uv-1)(u-1)(v-1)
\end{align*}
If we continue to repeat the skein relation in this manner, inductively we find:
\begin{align*}
	\Delta_0(VT(p(a_1),\dots,p(a_{i-1}),a_i,\dots,a_n))=&(-uv)^{\left \lfloor \frac{a_i}{2} \right \rfloor}\Delta_0(VT(p(a_1),\dots,p(a_i),a_{i+1},\dots,a_n))+\\&\left \lfloor \frac{a_i}{2} \right \rfloor (-uv)^{\sum_{j\geq i}{\left\lfloor \frac{a_j}{2} \right\rfloor}}(-1)^{\d+s(n)}(uv)^{\e(i)}(uv-1)(u-1)(v-1)
\end{align*}
Continuing inductively, we find:
\begin{align*}
\Delta_0(VT(a_1,\dots,a_n))(u,v)=&(-uv)^{\sum_{i=1}^{n}\left\lfloor\frac{a_i}{2}\right\rfloor}\Delta_0(VT(p(a_1),\dots,p(a_n))(u,v)+\\&(-uv)^{\sum_{i=1}^n{\left\lfloor \frac{a_i}{2} \right\rfloor}}(uv-1)(u-1)(v-1)(-1)^{\d+s(n)}\left ( \sum_{i=1}^{n}\left\lfloor\frac{a_i}{2}\right\rfloor(uv)^{\e(i)} \right )
\end{align*}

Then, after dividing through by $(u-1)(v-1)(uv-1)$, our formula becomes:

$$\overline{\Delta}_0(VT(a_1,\dots,a_n))(u,v)=(-uv)^{\sum_{i=1}^{n}\left\lfloor\frac{a_i}{2}\right\rfloor}\left(\overline{\Delta}_0(VT(p(a_1),\dots,p(a_n))(u,v)+(-1)^{\d+s(n)}\sum_{i=1}^{n}\left\lfloor\frac{a_i}{2}\right\rfloor(uv)^{\e(i)}\right)$$
as desired.
\end{proof}
\medskip

We now have an algorithm to compute $\Delta_0(VT(a_1,\dots,a_n))(u,v)$: \begin{enumerate}
	\item First, use Theorem \ref{T:recursion} to reduce the computation to finding $\overline{\Delta}_0(VT(p(a_1),\dots,p(a_n))$.
	\item If $p(a_i)=0$ for any $i$ with $1<i<n$, then use a Reidemeister (II*) move to remove the two consecutive virtual crossings (we will refer to this in the future as a contraction). Do this for all such $i$.  This may result in longer blocks of classical crossings.
	\item Check if we've reduced $VT(p(a_1),\dots,p(a_n))$ to one of the following: $VT(1,\dots,1)$, $VT(0,1,\dots,1)$, $VT(1,\dots,1,0)$, or $VT(0,1,\dots,1,0)$. If we haven't, we return to the first step. If we have, then we apply Theorem \ref{T:base}.
	\item Finally, normalize the result to find the Alexander polynomial of the virtual knot, rather than simply the diagram.
\end{enumerate}

This process will terminate because every every iteration reduces the number of classical and/or virtual crossings in the twist, unless we are at one of our base cases.  To show how to use the process, we will compute a few examples.

\begin{example}\label{E:1block}
The knot $VT(k)$ (with $k \geq 0$) is the twist knot with $k$ (positive) classical crossings in the twist, and no virtual crossings (except in the clasp).  Abusing notation, let $VT(k)$ also denote the diagram for this knot following the pattern of Figure \ref{F:virtualtwist}.
$$\overline{\Delta}_0(VT(k))(u,v) = (-uv)^{\left\lfloor\frac{k}{2}\right\rfloor}\left(\overline{\Delta}_0(VT(p(k)))(u,v) + \left\lfloor\frac{k}{2}\right\rfloor(-1)^{\d+s(1)}(uv)^{\e(1)}\right)$$
In this case, $s(1) = k+1$, $\d = p(k)p(k+1)=0$ and $\e(1) = -1+p(k)p(1+0)=p(k)-1$.  Also, from Theorem \ref{T:base}, $\overline{\Delta}_0(VT(0))(u,v) = 0$ and $\overline{\Delta}_0(VT(1))(u,v) = 1$, which means $\overline{\Delta}_0(VT(p(k)))(u,v) = p(k)$. So:
\begin{align*}
\overline{\Delta}_0(VT(k))(u,v) &= (-uv)^{\left\lfloor\frac{k}{2}\right\rfloor}\left(p(k) + \left\lfloor\frac{k}{2}\right\rfloor(-1)^{k+1}(uv)^{p(k)-1}\right) \\
&= \left\{\begin{matrix} \frac{k}{2}(-uv)^{\frac{k}{2} - 1} & k\ {\rm even}\\ (-uv)^{\left\lfloor\frac{k}{2}\right\rfloor}\left(\left\lfloor\frac{k}{2}\right\rfloor + 1\right) & k\ {\rm odd} \end{matrix} \right.
\end{align*}
If we normalize this to find the Alexander polynomial for the virtual knot $VT(k)$ (rather than just the diagram), we get
$$\overline{\Delta}_0(VT(k))(u,v) = \left\{\begin{matrix} \frac{k}{2} & k\ {\rm even}\\ \left\lfloor\frac{k}{2}\right\rfloor + 1 & k\ {\rm odd} \end{matrix} \right.$$
\end{example} \medskip

\begin{example}\label{E:2blocks}
The knot $VT(a,b)$ (with $a, b \geq 0$) has two blocks of (positive) crossings in the twist, separated by a virtual crossing.  As in the previous example, we will also let $VT(a,b)$ denote the diagram for this knot shown in Figure \ref{F:virtualtwist}.  From Theorem \ref{T:recursion} we have
$$\overline{\Delta}_0(VT(a,b))(u,v) = (-uv)^{\left\lfloor\frac{a}{2}\right\rfloor + \left\lfloor\frac{b}{2}\right\rfloor}\left(\overline{\Delta}_0(VT(p(a), p(b)))(u,v) + (-1)^{\d+s(2)}\left(\left\lfloor\frac{a}{2}\right\rfloor(uv)^{\e(1)}+\left\lfloor\frac{b}{2}\right\rfloor(uv)^{\e(2)}\right)\right)$$
Observe that
\begin{align*}
	\d &= p(a)p(a+1) + p(b)p(a+b+2) = p(b)p(a+b) \\
	s(2) &= a+b+2\\
	\e(1) &= -1 + p(a)p(1) + p(b)p(a+2) = p(a)(p(b)+1) - 1 \\
	\e(2) &= (p(a+1)-1) + p(a)p(a+1) + p(b)p(a+2) = -p(a) + p(b)p(a) = p(a)(p(b) - 1) 
\end{align*}
Also, from Theorem \ref{T:base} we have $\overline{\Delta}_0(VT(1,1))(u,v) = 1+u+uv$ and $\overline{\Delta}_0(VT(0,1))(u,v) = \overline{\Delta}_0(VT(1,0))(u,v) = \overline{\Delta}_0(VT(0,0))(u,v) = 0$.  Combining these, we find:
$$\overline{\Delta}_0(VT(a,b))(u,v) = (-uv)^{\left\lfloor\frac{a}{2}\right\rfloor + \left\lfloor\frac{b}{2}\right\rfloor}\cdot \left\{ 
\begin{matrix} 
(1 + u + uv) + \left(\left\lfloor\frac{a}{2}\right\rfloor(uv)^{1}+\left\lfloor\frac{b}{2}\right\rfloor(uv)^{0}\right) & {\rm \ if\ }a, b{\rm \ odd} \\
\left\lfloor\frac{a}{2}\right\rfloor(uv)^{-1}+\left\lfloor\frac{b}{2}\right\rfloor(uv)^{0} & {\rm \ if\ }a{\rm \ even}, b{\rm \ odd}\rule{0pt}{12pt} \\
(-1)\left(\left\lfloor\frac{a}{2}\right\rfloor(uv)^{0}+\left\lfloor\frac{b}{2}\right\rfloor(uv)^{-1}\right) & {\rm \ if\ }a{\rm \ odd}, b{\rm \ even}\rule{0pt}{12pt} \\
\left\lfloor\frac{a}{2}\right\rfloor(uv)^{-1}+\left\lfloor\frac{b}{2}\right\rfloor(uv)^{0} & {\rm \ if\ }a, b{\rm \ even}\rule{0pt}{12pt}
\end{matrix} \right.$$
Once we normalize the polynomial, we get:
$$\overline{\Delta}_0(VT(a,b))(u,v) = \left\{ \begin{matrix} 
\left(\left\lfloor\frac{b}{2}\right\rfloor + 1\right) +u + \left(\left\lfloor\frac{a}{2}\right\rfloor + 1\right)uv & {\rm \ if\ }a, b{\rm \ odd} \\
\left\lfloor\frac{a}{2}\right\rfloor+\left\lfloor\frac{b}{2}\right\rfloor uv & {\rm \ if\ }a{\rm \ even}, b{\rm \ odd}\rule{0pt}{12pt} \\
\left\lfloor\frac{b}{2}\right\rfloor+\left\lfloor\frac{a}{2}\right\rfloor uv & {\rm \ if\ }a{\rm \ odd}, b{\rm \ even}\rule{0pt}{12pt} \\
\left\lfloor\frac{a}{2}\right\rfloor+\left\lfloor\frac{b}{2}\right\rfloor uv & {\rm \ if\ }a, b{\rm \ even}\rule{0pt}{12pt}
\end{matrix} \right.$$
\end{example} \medskip

\begin{example}\label{E:VT(7,4,3,5,9)}
Consider the knot $VT(7,4,3,5,9)$, with the diagram from Figure \ref{F:virtualtwist}.  First observe that $\d = 3$, $s(5) = 33$, $\e(1) = 0$, $\e(2) = -1$, $\e(3) = 0$, $\e(4) = 1$ and $\e(5) = 2$.  So by Theorem \ref{T:recursion}, we have:
\begin{align*}
\overline{\Delta}_0(VT(7,4,3,5,9))(u,v) &= (-uv)^{12}\left(\overline{\Delta}_0(VT(1,0,1,1,1))(u,v) + 4 + 2(uv)^{-1} + 2(uv) + 4(uv)^2 \right)\\
&= (-uv)^{12}\left(\overline{\Delta}_0(VT(2,1,1))(u,v) + 4 + 2(uv)^{-1} + 2(uv) + 4(uv)^2 \right)\\
&= (-uv)^{12}\left((-uv)\left(\overline{\Delta}_0(VT(0,1,1))(u,v) - (uv)^{-1} \right) + 4 + 2(uv)^{-1} + 2(uv) + 4(uv)^2 \right) \\
&= (-uv)^{12}\left((-uv)\left(v - (uv)^{-1} \right) + 4 + 2(uv)^{-1} + 2(uv) + 4(uv)^2 \right) \\
&= (-uv)^{12}\left(-uv^2 + 5 + 2(uv)^{-1} + 2(uv) + 4(uv)^2 \right)
\end{align*}
If we normalize this to get the invariant for the virtual knot, rather than the diagram, we get:
$$\overline{\Delta}_0(VT(7,4,3,5,9))(u,v) = 2 + 5uv - u^2v^3 + 2u^2v^2 + 4u^3v^3$$
\end{example}

\section{Negative crossings} \label{S:negative}

If some of the blocks of crossings are negative, we need to modify the algorithm outlined in the last section.  Theorem \ref{T:recursion} still holds, with the caveats that each appearance of $a_i$ is replaced by $\abs{a_i}$, the term $\left\lfloor\frac{\abs{a_i}}{2}\right\rfloor(uv)^{\e(i)}$ is {\em subtracted} whenever $a_i < 0$, and the twist knot is reduced to $VT(\e_1p(a_1), \dots, \e_np(a_n))$, where $\e_i$ is the sign of the crossings in the $i$th block.  In step (2) of the algorithm, if we combine a block of positive and negative crossings, we can cancel some of the crossings using Reidemeister (II) moves, at the expense of multiplying the polynomial by a factor of $-uv$ for each Reidemeister (II) move.  The end result of the algorithm will be one of $VT(\pm 1,\dots,\pm 1)$, $VT(0,\pm 1,\dots,\pm1)$, $VT(\pm1,\dots,\pm1,0)$, or $VT(0,\pm1,\dots,\pm1,0)$.

We can now use our skein relation and Lemma \ref{L:smoothed} to change any negative crossings to positive crossings, resulting in one of the base cases in Section \ref{S:base}.  The proof of the following result is straightforward, and is left to the reader.  (In the cases where the first block is $0$, rather than $\pm 1$, the crossing being changed is type 2, and the sign in the skein relation changes.)

\begin{cor} \label{C:negative}
In each case below, the twist knot has $n$ blocks, each with a single crossing (except possibly the first and/or last), and the crossing in the $i$th block is negative.
\begin{align*}
\overline{\Delta}_0(VT(\pm 1,\dots, -1, \dots,\pm 1)) &= \overline{\Delta}_0(VT(\pm 1,\dots, 1, \dots,\pm 1)) - (uv)^{n-i} \\
\overline{\Delta}_0(VT(0, \pm 1,\dots, -1, \dots,\pm 1)) &= \overline{\Delta}_0(VT(0, \pm 1,\dots,1, \dots,\pm 1)) - (-1)^n(uv)^{i-2} \\
\overline{\Delta}_0(VT(\pm 1,\dots, -1, \dots,\pm 1, 0)) &= \overline{\Delta}_0(VT(\pm 1,\dots, 1, \dots,\pm 1, 0)) + (uv)^{n-i-1} \\
\overline{\Delta}_0(VT(0, \pm 1,\dots, -1,\dots,\pm 1, 0)) &= \overline{\Delta}_0(VT(0, \pm 1,\dots, 1, \dots,\pm 1, 0)) - (-1)^n(uv)^{i-2}
\end{align*}
\end{cor}

\begin{example}
We will compute $\overline{\Delta}_0$ for $VT(-7, 3, -5, -2, 3)$.  Observe that $\d = 1$, $s(5) = -3$, $\e(1) = 2$, $\e(2) = 1$, $\e(3) = 0$, $\e(4) = -1$ and $\e(5) = 0$. So by Theorem \ref{T:recursion}, we have:
\begin{align*}
\overline{\Delta}_0(VT(-7, 3, -5, -2, 3))(u,v) &= (-uv)^{8}\left(\overline{\Delta}_0(VT(-1,1,-1,0,1))(u,v) - 3(uv)^2 + uv - 2 - (uv)^{-1} + 1 \right)\\
&= (-uv)^{8}\left((-uv)\overline{\Delta}_0(VT(-1,1,0))(u,v) - 3(uv)^2 + uv - (uv)^{-1} - 1 \right)\\
&= (-uv)^{8}\left((-uv)\left(\overline{\Delta}_0(VT(1,1,0))(u,v)+uv\right) - 3(uv)^2 + uv - (uv)^{-1} - 1 \right)\\
&= (-uv)^{8}\left((-uv)(u+uv) - 3(uv)^2 + uv - (uv)^{-1} - 1 \right)\\
&= (-uv)^{8}\left(-u^2v-(uv)^2 - 3(uv)^2 + uv - (uv)^{-1} - 1 \right)\\
&= (-uv)^{8}\left(-u^2v - 4(uv)^2 + uv - (uv)^{-1} - 1 \right)
\end{align*}
Normalizing to get the invariant for the virtual knot, we get:
$$\overline{\Delta}_0(VT(-7, 3, -5, -2, 3))(u,v) = 1 + uv - u^2v^2 + u^3v^2 + 4u^3v^3$$
\end{example}

\section{Other virtual twist knots} \label{S:clasp}

\begin{figure}[htbp]
\begin{center}
\scalebox{.8}{\includegraphics{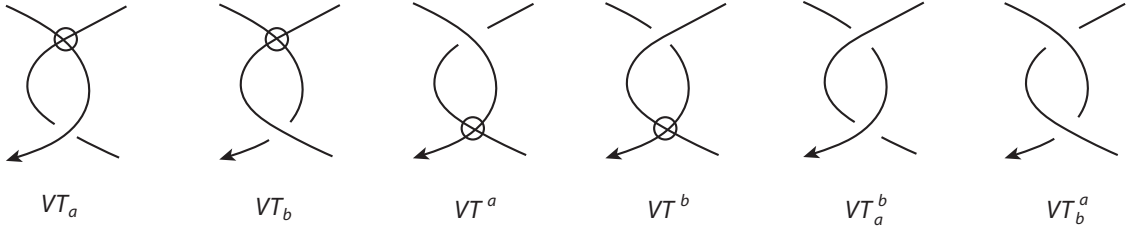}}
\end{center}
\caption{Possible clasps for virtual twist knots.}
\label{F:clasps}
\end{figure}

The choice of the clasp in the virtual twist knot $VT(a_1, \dots, a_n)$ shown in Figure \ref{F:virtualtwist} is somewhat arbitrary.  We could well have picked any of the other clasps shown in Figure \ref{F:clasps} (note that the clasp denoted $VT_a$ is the one we have been using).  In this section we will briefly discuss these other options.

\begin{figure}[htbp]
\begin{center}
\scalebox{.8}{\includegraphics{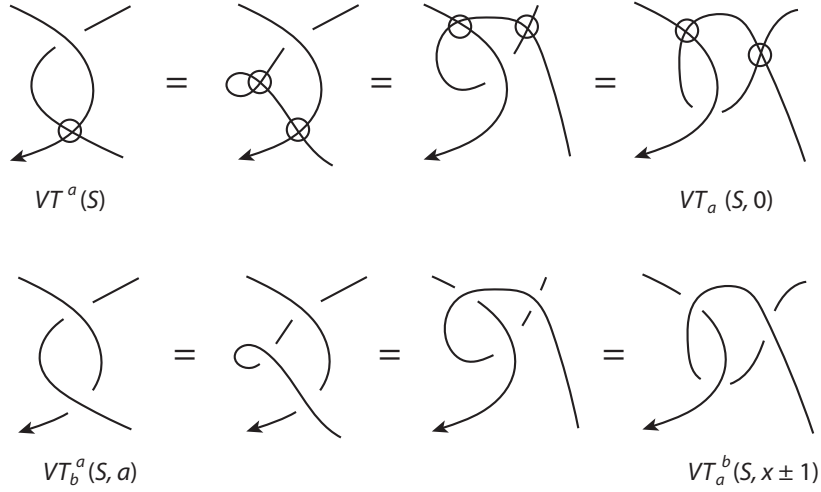}}
\end{center}
\caption{Equivalences among clasps.}
\label{F:VTa}
\end{figure}

We first observe that some of these clasps are really equivalent.  If $S = (a_1, \dots, a_n)$ is a string of twists, then we can use Reidemeister moves to show that $VT^a(S) = VT_a(S, 0)$ and $VT_b^a(S, x) = VT_a^b(S, x \pm 1)$ (see Figure \ref{F:VTa}).  Also, $VT_b(S) = VT_a(-S)^\#$ (where $-S$ indicates that we have reversed the sign of each block of crossings in the twist) and $VT^b(S) = VT^a(-S)^\#$.  Then by Proposition \ref{P:reverse} we have:
\begin{align*}
	\Delta_0(VT^a(S))(u,v) &= \Delta_0(VT_a(S, 0))(u,v) \\
	\Delta_0(VT_b(S))(u,v) &= -\Delta_0(VT_a(-S))(v,u) \\
	\Delta_0(VT^b(S))(u,v) &= -\Delta_0(VT_a(-S, 0))(v,u) \\
	\Delta_0(VT_b^a(S,x))(u,v) &= \Delta_0(VT_a^b(S, x\pm 1))(u,v) 
\end{align*}

So the only really new class of virtual knots (as far as $\Delta_0$ is concerned) is $VT_a^b(S)$.  We can compute the base cases and recursion formula much as we did for $VT(S)$; in fact, since classical Hopf links have $\Delta_0 = 0$, the recursion is even simpler in this case.  The next theorem gives the results of these computations; the details are left to the reader.

\begin{thm}\label{T:VTab}
Given a sequence of $m$ ones, then:
\begin{align*}
		\Delta_0(VT_a^b(1,\dots,1)) &= uv(u-1)(v-1)(uv-1)\sum_{i=0}^{m-2}\sum_{j=0}^{m-2}{u^iv^j} {\rm \ (or\ 0\ if\ }m\leq 1) \\
		\Delta_0(VT_a^b(0,1,\dots,1)) &= (-1)^{m}uv(u-1)(v-1)(uv-1)\sum_{i=0}^{m-1}\sum_{j=0}^{m-1}{u^{i}v^{j}} {\rm \ (or\ 0\ if\ }m = 0) \\
		\Delta_0(VT_a^b(1,\dots,1,0)) &= uv(u-1)(v-1)(uv-1)\sum_{i=0}^{m-1}\sum_{j=0}^{m-1}{u^{i}v^{j}} {\rm \ (or\ 0\ if\ }m = 0) \\
		\Delta_0(VT_a^b(0,1,\dots,1,0)) &= (-1)^{m}(u-1)(v-1)(uv-1)\sum_{i=0}^{m}\sum_{j=0}^{m}{u^{i}v^{j}} 
\end{align*}
Moreover, if all crossings in the twist are positive (negative crossings can be dealt with as in Section \ref{S:negative}), then
$$\overline{\Delta}_0(VT_a^b(a_1, \dots, a_n))(u,v) = (-uv)^{\sum_{i=1}^{n}\left\lfloor\frac{a_i}{2}\right\rfloor}\overline{\Delta}_0(VT_a^b(p(a_1),\dots,p(a_n)))(u,v)$$
After normalizing, this implies
$$\Delta_0(VT_a^b(a_1, \dots, a_n))(u,v) = \Delta_0(VT_a^b(p(a_1),\dots,p(a_n)))(u,v)$$
\end{thm}

\begin{example} \label{E:VTab}
As an example, we will compute $\Delta_0(VT_a^b(x,y))$, where $x, y \geq 0$.  By the recursion, after normalizing, we have $\Delta_0(VT_a^b(x,y)) = \Delta_0(VT_a^b(p(x), p(y)))$.  If we normalize by powers of $uv$ and the sign, we find that $\Delta_0(VT_a^b(x,y)) = 1$ for all choices of $x$ and $y$.
\end{example}

\section{Alexander polynomial and the odd writhe} \label{S:oddwrithe}

The {\em odd writhe} was introduced by Kauffman \cite{ka2}, who proved it is an invariant of virtual knots. A classical crossing of a virtual knot is called {\em odd} if the path from the crossing back to itself goes through an odd number of classical crossings. The {\em odd writhe} is the sum of the signs of the odd crossings.  In classical knots the odd writhe is always zero, since all crossings of classical knots are even.  In this section, we will prove that the odd writhe of a virtual twist knot $K$ is related to its Alexander polynomial; we conjecture that this relationship holds in general.  We first verify the relationship for our ``base cases," and then use the recursion of Theorem \ref{T:recursion} to extend it to other virtual twist knots.

\begin{lem}\label{L:owtwist}
The odd writhe of the virtual twist knot $VT(a_1, \dots, a_n)$ is $\sum{a_i} + p(\sum{a_i})(-1)^{s(n)}$.
\end{lem}
\begin{proof}
All the crossings in the twist are odd, and the crossing in the clasp is odd exactly when the total number of crossings in the twist is odd.  The sign of the crossing in the clasp is given by $(-1)^{s(n)}$.
\end{proof}

\begin{lem}\label{L:owbase}
For the virtual twist knot diagrams $VT(1,\dots,1)$, $VT(0,1,\dots,1)$, $VT(1,\dots,1,0)$, and $VT(0,1,\dots,1,0)$ (following the pattern of Figure \ref{F:virtualtwist}), we have:
$$2\overline{\Delta}_0(VT(a_1,\dots,a_n))(-1,-1) = (-1)^{\delta+s(n)+\sum_{i=1}^{n}\left\lfloor\frac{a_i}{2}\right\rfloor}OW(VT(a_1,\dots,a_n))$$
where $OW$ denotes the odd writhe.
\end{lem}

\begin{proof}

{\sc Case 1: } Given $VT(1,\dots,1)$ with $m$ classical crossings in the twist, $OW(VT(1,\dots,1))=m+p(m)$ by Lemma \ref{L:owtwist} (since $s(m) = 2m$).

By Theorem \ref{T:base},
$$\overline{\Delta}_0(VT(1,\dots,1))(-1,-1)=\sum_{i=0}^{m-1}(-1)^i\sum_{j=i}^{m-1}(-1)^{j}.$$
Notice that:
	$$
		\sum_{j=i}^{m-1}(-1)^{j}= \begin{cases}
0, & m-1\text{ and }i\text{ have opposite parity}\\ 
1, & m-1\text{ and }i\text{ are both even}\\ 
-1, & m-1\text{ and }i\text{ are both odd} 
\end{cases}
	$$
	Therefore, we see that:
	$$
		\sum_{i=0}^{m-1}(-1)^{i}\sum_{j=i}^{m-1}(-1)^{j}=\begin{cases}
		\frac{m}{2}, & m-1\text{ is odd}\\ 
\frac{m+1}{2}, & m-1\text{ is even}
\end{cases}=\left\lceil\frac{m}{2}\right\rceil
	$$
So in this case:
$$2\overline{\Delta}_0(VT(1,\dots,1))(-1,-1) = 2\left\lceil\frac{m}{2}\right\rceil = m + p(m) = OW(VT(1,\dots,1))$$
Now, notice that:
	$$
		p(s(i))=p\left ( \sum_{j=1}^{i}(1+1) \right )=p(2i)=0	
	$$
	for any $1\leq i\leq m$. Therefore:
	$$
		\delta=\sum_{i=1}^{m}p(a_i)p(s(i))=0	
	$$
	We also have that $s(m) = 2m$ and $\sum_{i=1}^{m}\left\lfloor\frac{a_i}{2}\right\rfloor = 0$, so we see that $(-1)^{\delta+s(m)+\sum_{i=1}^{m}\left\lfloor\frac{a_i}{2}\right\rfloor}=(-1)^{2m}=1$, and we get the desired result. \medskip

\noindent{\sc Case 2: } Given $VT(0,1,\dots,1)$ with $m$ classical crossings in the twist, $OW(VT(1,\dots,1))=m-p(m)$ by Lemma \ref{L:owtwist} (since $s(m+1) = 2m+1$).

By Theorem \ref{T:base}, 
$$\overline{\Delta}_0(VT(0, 1,\dots,1))(-1,-1)=(-1)^{m+1}\sum_{i=0}^{m-2}(-1)^i\sum_{j=i}^{m-2}(-1)^{j}.$$
As in the previous case,
	$$
		\sum_{i=0}^{m-2}(-1)^{i}\sum_{i=j}^{m-2}(-1)^{j}=\begin{cases}
		\frac{m-1}{2}, & m\text{ is odd}\\ 
\frac{m}{2}, & m\text{ is even}
\end{cases}=\left\lfloor\frac{m}{2}\right\rfloor
	$$
So in this case:
$$2\overline{\Delta}_0(VT(0, 1,\dots,1))(-1,-1) = 2(-1)^{m+1}\left\lfloor\frac{m}{2}\right\rfloor = (-1)^{m+1}(m - p(m)) = (-1)^{m+1}OW(VT(0, 1,\dots,1))$$
Now, notice that:
	$$
		p(s(i))=p(2(i-1)+1)=1
	$$
	for any $1\leq i\leq m+1$. Therefore:
	$$
		\delta=\sum_{i=1}^{m+1}p(a_i)p(s(i))=\sum_{i=2}^{m+1}1=m
	$$
So $(-1)^{\delta+s(m+1)+\sum_{i=1}^{m+1}\left\lfloor\frac{a_i}{2}\right\rfloor}=(-1)^{m+(2m+1)+0}=(-1)^{3m+1}=(-1)^{m+1}$. That is exactly the sign difference between $2\overline{\Delta}_0(VT(0,1,\dots,1))(-1,-1)$ and $OW(VT(0,1,\dots,1))$. \medskip
	
\noindent{\sc Case 3: } Given $VT(1,\dots,1,0)$ with $m$ classical crossings in the twist, $OW(VT(1,\dots,1))=m-p(m)$ by Lemma \ref{L:owtwist} (since $s(m+1) = 2m+1$).

By Theorem \ref{T:base}, 
$$\overline{\Delta}_0(VT(1,\dots,1,0))(-1,-1)=(-1)\sum_{i=0}^{m-2}(-1)^i\sum_{j=i}^{m-2}(-1)^{j} = (-1)\left\lfloor\frac{m}{2}\right\rfloor$$
So in this case:
$$2\overline{\Delta}_0(VT(1,\dots,1,0))(-1,-1) = 2(-1)\left\lfloor\frac{m}{2}\right\rfloor = (-1)(m - p(m)) = (-1)OW(VT(1,\dots,1,0))$$
Now, notice that:
	$$
		p(s(i))=p\left ( \sum_{j=1}^{i}(1+1) \right )=p(2i)=0	
	$$
	for any $1\leq i\leq m$. Also, $p(s(m+1))=1$. Therefore:
	$$
		\delta=\sum_{i=1}^{m+1}p(a_i)p(s(i))=p(a_{m+1})p(s(m+1))=0
	$$
So $(-1)^{\delta+s(m+1)+\sum_{i=1}^{m+1}\left\lfloor\frac{a_i}{2}\right\rfloor}=(-1)^{2m+1}=-1$. That is exactly the sign difference between\\ $2\overline{\Delta}_0(VT(1,\dots,1,0))(-1,-1)$ and $OW(VT(1,\dots,1,0))$.\medskip

\noindent{\sc Case 4: } Given $VT(0,1,\dots,1,0)$ with $m$ classical crossings in the twist, $OW(VT(1,\dots,1))=m+p(m)$ by Lemma \ref{L:owtwist} (since $s(m+2) = 2m+2$).

By Theorem \ref{T:base}, 
$$\overline{\Delta}_0(VT(0,1,\dots,1,0))(-1,-1)=(-1)^{m}\sum_{i=0}^{m-1}(-1)^i\sum_{j=i}^{m-1}(-1)^{j}= (-1)^m\left\lceil\frac{m}{2}\right\rceil$$
So in this case:
$$2\overline{\Delta}_0(VT(0,1,\dots,1,0))(-1,-1) = 2(-1)^m\left\lceil\frac{m}{2}\right\rceil = (-1)^m(m + p(m)) = (-1)^mOW(VT(0,1,\dots,1,0))$$
Now, notice that:
	$$
		p(s(i))=p\left ( (0+1)+\sum_{j=2}^{i}(1+1) \right )=p(2(i-1)+1)=1	
	$$
	for any $1\leq i\leq m+1$. Also, $p(s(m+2))=0$. Therefore:
	$$
		\delta=\sum_{i=1}^{m+2}p(a_i)p(s(i))=m
	$$
So $(-1)^{\delta+s(m+2)+\sum_{i=1}^{m+2}\left\lfloor\frac{a_i}{2}\right\rfloor}=(-1)^{3m+2} = (-1)^m$. That is exactly the sign difference between\\ $2\overline{\Delta}_0(VT(0,1,\dots,1,0))(-1,-1)$ and $OW(VT(0,1,\dots,1,0))$.
\end{proof}

Now we will use the recursion formula of Theorem \ref{T:recursion} to generalize this relationship to all virtual twist knots.

\begin{thm}\label{T:oddwrithe}
For any virtual twist knot diagram $VT(a_1, \dots, a_n)$ (following the pattern of Figure \ref{F:virtualtwist}), we have:
$$2\overline{\Delta}_0(VT(a_1,\dots,a_n))(-1,-1) = (-1)^{\delta+s(n)+\sum_{i=1}^{n}\left\lfloor\frac{a_i}{2}\right\rfloor}OW(VT(a_1,\dots,a_n))$$
\end{thm}
\begin{proof}
We have shown that we can reduce any virtual twist knot $VT(a_1, \dots, a_n)$ to one of $VT(1,\dots,1)$, $VT(0,1\dots,1)$, $VT(1,\dots,1,0)$, or $VT(0,1,\dots,1,0)$ by successive applications of (i) the recursion formula in Theorem \ref{T:recursion} and (ii) contractions (removing blocks with 0 crossings).  By Lemma \ref{L:owbase}, we know that our relation holds for the four base cases. It remains to show that it is preserved by both the recursive operation and by contraction.

First we will show it is preserved by the recursion.  Inductively, we assume that 
$$2\overline{\Delta}_0(VT(p(a_1),\dots,p(a_n)))(-1,-1)=(-1)^{\delta+s(n)+\sum_{i=1}^{n}\left\lfloor\frac{p(a_i)}{2}\right\rfloor}OW(VT(p(a_1),\dots,p(a_n))).$$
Notice that $\d$ is the same for both $VT(a_1, \dots, a_n)$ and $VT(p(a_1), \dots, p(a_n))$, and that $s(n)$ has the same sign for both, so these two terms in the power of $-1$ are the same for both knots.  Now, using Theorem \ref{T:recursion}, we have:
\begin{align*}
2\overline{\Delta}_0(VT(a_1,&\dots,a_n))(-1,-1)=2(-1)^{\sum_{i=1}^{n}\left\lfloor\frac{a_i}{2}\right\rfloor}\left ( \overline{\Delta}_0(VT(p(a_1),\dots,p(a_n)))(-1,-1)+\sum_{i=1}^{n}\left\lfloor\frac{a_i}{2}\right\rfloor(-1)^{\delta+s(n)} \right )\\
&=(-1)^{\delta+s(n)+\sum_{i=1}^{n}\left\lfloor\frac{a_i}{2}\right\rfloor+\sum_{i=1}^{n}\left\lfloor\frac{p(a_i)}{2}\right\rfloor}OW(VT(p(a_1),\dots,p(a_n)))+2\sum_{i=1}^{n}\left\lfloor\frac{a_i}{2}\right\rfloor(-1)^{\delta+s(n)+\sum_{i=1}^{n}\left\lfloor\frac{a_i}{2}\right\rfloor}\\
&=(-1)^{\delta+s(n)+\sum_{i=1}^{n}\left\lfloor\frac{a_i}{2}\right\rfloor}\left ( (-1)^{\sum_{i=1}^{n}\left\lfloor\frac{p(a_i)}{2}\right\rfloor}\left ( \sum_{i=1}^{n}p(a_i)+(-1)^{s(n)}p\left ( \sum_{i=1}^{n}p(a_i) \right ) \right )+\sum_{i=1}^{n}2\left\lfloor\frac{a_i}{2}\right\rfloor \right )\\
&=(-1)^{\delta+s(n)+\sum_{i=1}^{n}\left\lfloor\frac{a_i}{2}\right\rfloor}\left ( \sum_{i=1}^{n}\left ( p(a_i)+(a_i - p(a_i)) \right )+(-1)^{s(n)}p\left ( \sum_{i=1}^{n}a_i \right ) \right)\\
&=(-1)^{\delta+s(n)+\sum_{i=1}^{n}\left\lfloor\frac{a_i}{2}\right\rfloor}\left ( \sum_{i=1}^{n}a_i+(-1)^{s(n)}p\left ( \sum_{i=1}^{n}a_i \right ) \right )\\
&=(-1)^{\delta+s(n)+\sum_{i=1}^{n}\left\lfloor\frac{a_i}{2}\right\rfloor}OW(VT(a_1,\dots,a_n))
\end{align*}

Now we will show that the relationship is preserved by contraction.  Consider the two virtual twist knots $VT(a_1,\dots,a_n)$ and $VT(a_1,\dots,b,0,a_i-b,\dots,a_n)$ (where $b<a_i$), with diagrams following the pattern in Figure \ref{F:virtualtwist}.  These two knots differ only by two consecutive virtual crossings in the twist.  Removing these two crossings by a Reidemister (II*) move does not change the labeling of the diagram, so both diagrams have the same value of $\Delta_0$.  Also, the two knots are equivalent, so they have the same odd writhe.  The only thing that remains to check is that the power $\delta+s(n)+\sum_{j=1}^{n}\left\lfloor\frac{a_j}{2}\right\rfloor$ has the same sign for both knot diagrams.  Then if the relation holds for $VT(a_1,\dots,b,0,a_i-b,\dots,a_n)$, it will also hold for the contraction 
$VT(a_1,\dots,a_n)$.

Now, we let $VT(a_1, \dots, b, 0, a_i-b, \dots, a_n) = VT(a_1', \dots, a_{n+2}')$, where:
$$a_j' = \begin{cases} a_j, & j < i \\ b, & j = i\\ 0, & j = i+1\\ a_i - b, & j = i+2\\ a_{j-2}, & j > i+2 \end{cases}$$

Recall that $s(j) = \sum_{k=1}^j{a_j+1}$ and $\d = \sum_{j=1}^n{p(a_j)p(s(j))}$ for $VT(a_1,\dots,a_n)$.  Let $s'(j)$ and $\d'$ be the corresponding sums for $VT(a_1', \dots, a_{n+2}')$.  Then,
$$s'(j) = \begin{cases} s(j) & j < i \\ s(i-1)+b+1, & j = i\\ s(i-1)+b+2, & j = i+1\\ s(j-2)+2, & j \geq i+2 \end{cases}$$
We can use this to compute $\d'$.
\begin{align*}
	\delta' &= \sum_{j=1}^n{p(a_j')p(s'(j))} \\
	&=\sum_{j=1}^{i-1}p(a_j)p(s(j))+p(b)p(s(i-1)+b+1)+p(0)p(s(i-1)+b+2)+p(a_i-b)p(s(i)+2)+\sum_{j=i+1}^{n}p(a_j)p(s(j))\\
	&=\d - p(a_i)p(s(i)) + p(b)p(s(i-1)+b+1)+p(a_i-b)p(s(i)+2)
\end{align*}
Finally, we observe that 
$$\sum_{j=1}^{n+2}\left\lfloor\frac{a_j'}{2}\right\rfloor=\sum_{j=1}^{i-1}\left\lfloor\frac{a_j}{2}\right\rfloor+\left\lfloor\frac{b}{2}\right\rfloor+\left\lfloor\frac{0}{2}\right\rfloor+\left\lfloor\frac{a_i-b}{2}\right\rfloor+\sum_{j=i+1}^{n}\left\lfloor\frac{a_j}{2}\right\rfloor=\sum_{j=1}^{n}\left\lfloor\frac{a_j}{2}\right\rfloor-\left\lfloor\frac{a_i}{2}\right\rfloor+\left\lfloor\frac{b}{2}\right\rfloor+\left\lfloor\frac{a_i-b}{2}\right\rfloor
$$
Since $s'(n+2) = s(n)+2$, this does not change the parity of the power of $-1$.  It remains to show that 
\begin{align*}
D &= \left(\d' + \sum_{j=1}^{n+2}\left\lfloor\frac{a_j'}{2}\right\rfloor\right) - \left(\d + \sum_{j=1}^{n}\left\lfloor\frac{a_j}{2}\right\rfloor\right) \\
&= - p(a_i)p(s(i)) + p(b)p(s(i-1)+b+1)+p(a_i-b)p(s(i)+2)-\left\lfloor\frac{a_i}{2}\right\rfloor+\left\lfloor\frac{b}{2}\right\rfloor+\left\lfloor\frac{a_i-b}{2}\right\rfloor
\end{align*}
is even.  In fact, $D$ is always 0.  To verify this, we must check four cases, each corresponding to the different choices of parity for $a_i$ and $b$. \medskip

\noindent{\sc Case 1:} Assume $a_i$ is even and $b$ is even. Then $a_i-b$ is also even, and $D = 0$. \medskip
	
\noindent{\sc Case 2:} Assume $a_i$ is even and $b$ is odd. Then $a_i - b$ is also odd.  So:
\begin{align*}
	D &= p(s(i-1)+b+1) + p(s(i)+2) - \frac{a_i}{2} + \frac{b-1}{2} + \frac{a_i-b-1}{2}\\
	&= p(s(i-1)) + p(s(i)) - 1 \\
	&= p(s(i-1)) + p(s(i-1) + a_i+1) - 1\\
	&= p(s(i-1)) + p(s(i-1)+1) - 1 = 0
\end{align*} \medskip
	
\noindent{\sc Case 3:} Assume $a_i$ is odd and $b$ is even.  Then $a_i-b$ is odd.  So:
\begin{align*}
	D &= -p(s(i)) + p(s(i)+2) - \frac{a_i-1}{2} + \frac{b}{2} + \frac{a_i-b-1}{2}\\
	&= -p(s(i)) + p(s(i)) = 0 
\end{align*} \medskip
	
\noindent{\sc Case 4:} Assume $a_i$ is odd and $b$ is odd. Then $a_i-b$ is even.  So:
\begin{align*}
	D &= -p(s(i)) + p(s(i-1)+b+1) - \frac{a_i-1}{2} + \frac{b-1}{2} + \frac{a_i-b}{2}\\
	&= -p(s(i)) + p(s(i-1)) \\
	&= -p(s(i-1) + a_i+1) + p(s(i-1)) \\
	&= -p(s(i-1)) + p(s(i-1)) = 0
\end{align*} \medskip

So contraction also preserves the relationship, which completes the proof.	
\end{proof}

\begin{cor}\label{C:owtwist}
For any virtual twist knot $K$,
$$2\vert \Delta_0(K)(-1,-1) \vert = \vert OW(K)\vert$$
\end{cor}
\begin{proof}
We know that for any virtual twist knot $K$, there is some diagram for $K$ of the form $VT(a_1, \dots, a_n)$ shown in Figure \ref{F:virtualtwist}.  So $\Delta_0(K)(u,v) = (-1)^s(uv)^t\Delta_0(VT(a_1, \dots, a_n))$ for some integers $s$ and $t$.  But then:
$$2\vert \Delta_0(K)(-1,-1) \vert = 2\vert \Delta_0(VT(a_1, \dots, a_n))(-1,-1) \vert =\vert OW(VT(a_1, \dots, a_n))\vert =\vert OW(K)\vert$$
\end{proof}

We conjecture that this result extends to {\em all} virtual knots:

\begin{conj}
For any virtual knot $K$,
$$2\vert \Delta_0(K)(-1,-1) \vert = \vert OW(K)\vert$$
\end{conj}

\begin{rem}
In the same way as above, we can check that the conjecture holds for the other twist knots considered in Section \ref{S:clasp}.  We have also used a computer to verify the conjecture for the $364,253$ unoriented virtual knots with 6 or fewer crossings, as tabulated by Green \cite{gr}. By Proposition \ref{P:reverse}, changing the orientation will only change $\Delta_0(K)(-1,-1)$ by a sign.  Since the odd writhe is invariant under a change in orientation, this means the conjecture holds for all $725,854$ oriented virtual knots with 6 or fewer crossings.
\end{rem}\bigskip

\small

\normalsize

\end{document}